\documentclass[10pt]{amsart}

\usepackage[footnotesize,bf]{caption}
\usepackage[left=1.25in,right=1.25in,top=1in]{geometry}

\usepackage{amsmath,amssymb,bbm,amsthm}
\usepackage{graphicx}
\usepackage{verbatim}
\usepackage{mathrsfs,mathtools}

\usepackage{shortcuts}

\usepackage{array,multirow,booktabs} % for fancier tables

\renewcommand{\thefootnote}{\fnsymbol{footnote}}

\newcommand{\innerp}[1]{\langle {#1} \rangle}
\newcommand{\abs}[1]{\lvert#1\rvert}
\newcommand{\bs}[1]{\boldsymbol{#1}}

\newcommand{\bsk}{\boldsymbol{k}}

\newcommand{\bsj}{\boldsymbol{j}}
\newcommand{\bsn}{\boldsymbol{n}}

\newcommand{\bx}{\bs{x}}
\newcommand{\bz}{\bs{z}}

\def\vx{\mathbf x}

\def\vPhi{\mathbf \Phi}

\def\vc{\mathbf c}
\def\vv{\mathbf v}

\def\be{\begin{equation}}
\def\ee{\end{equation}}

\def\R{{\mathbb R}}
\def\N{{\mathbb N}}

\def\v{\widetilde{v}}

\newcommand{\normmm}[1]{{\left\vert\kern-0.25ex\left\vert\kern-0.25ex\left\vert #1
    \right\vert\kern-0.25ex\right\vert\kern-0.25ex\right\vert}}

\usepackage{chngcntr}

\newcounter{parentnumber}


\usepackage[colorlinks=true]{hyperref}
\usepackage{cleveref}
\crefname{theorem}{Theorem}{Theorems}

\title{A gradient enhanced $\ell_1$-minimization for sparse approximation of polynomial chaos expansions}

\author{Ling Guo}
\thanks{Ling Guo. Department of Mathematics, Shanghai Normal University, Shanghai, China. L.~Guo was partially suppoted by NSFC (11671265),
and Program for Outstanding Academic leaders in Shanghai City (No.151503100). Email: lguo@shnu.edu.cn.}

\author{Akil Narayan}
\thanks{Akil Narayan. Mathematics Department and Scientific Computing and Imaging Institute, University of Utah, University of Utah, Salt Lake City, UT 84112.
Email: akil@sci.utah.edu. A.~Narayan was partially supported by AFOSR FA9550-15-1-0467 and DARPA N660011524053}

\author{Tao Zhou}
\thanks{Tao Zhou. LSEC, Institute of Computational Mathematics and Scientific/Engineering Computing, AMSS, Chinese Academy of Sciences, Beijing, China.
Email: tzhou@lsec.cc.ac.cn. T.~Zhou is partially supported by the National Natural Science Foundations of China (under grant numbers 91630312, 91630203, 11571351, 11688101, and 11731006), the science challenge project (No. TZ2018001), NCMIS, and the youth innovation promotion association (CAS).}

\begin{document}
\maketitle
%\slugger{mms}{xxxx}{xx}{x}{x--x}%slugger should be set to mms, siap, sicomp, sicon, sidma, sima, simax, sinum, siopt, sisc, or sirev
%\graphicspath{{\detokenize{polynomial_results/}}{\detokenize{analytical_results/}}{\detokenize{pde_results/}}{\detokenize{MIP_results/}}}
\renewcommand{\thefootnote}{\fnsymbol{footnote}}

\begin{abstract}
We investigate a gradient-enhanced $\ell_1$-minimization for constructing sparse polynomial chaos expansions. In addition to function evaluations, measurements of the function gradient is also included to accelerate the identification of expansion coefficients. By designing appropriate preconditioners to the measurement matrix, we show gradient-enhanced $\ell_1$ minimization leads to stable and accurate coefficient recovery. The framework for designing preconditioners is quite general and it applies to recover of functions whose domain is bounded or unbounded. Comparisons between the gradient enhanced approach and the standard $\ell_1$-minimization are also presented and numerical examples suggest that the inclusion of derivative information can guarantee sparse recovery at a reduced computational cost.
\end{abstract}

%\begin{keywords}
%Polynomial chaos expansions, Uncertainty quantification, Compressive sampling, Gradient-enhanced
%$\ell_1$-minimization.
%\end{keywords}

%\begin{AMS}
%\end{AMS}

\pagestyle{myheadings}
\thispagestyle{plain}
\markboth{Gradient enhanced $\ell_1$-minimization}{}

\section{Introduction}
Uncertainty quantification (UQ) aims to develop numerical methods that can accurately approximate quantities of interest (QoI) of a complex engineering system and facilitate the quantitative validation of the simulation model. One challenge in UQ is in building surrogates for approximation of a parameterized simulation model, often involving differential equations. To characterize the uncertainty that parameters effect on such a system, one usually models the uncertain inputs as a $d$-dimensional vector of independent random variables $\bx=(x_1,\ldots,x_d)$. The QoI $f$ that we seek to approximate is a function of these random parameters, $f(\bx): \mathbb{R}^d\to \mathbb{R}$. Here we will approximate $f(\bx)$ with a generalized Polynomial Chaos Expansion (PCE) \cite{R.Ghanem1991_SFE,Xiu2002_WPC}. In this situation, we assume $f$ can be well-approximated as a finite expansion in multivariate orthogonal polynomials, and the key step is to determine the expansion coefficients.

Recently, stochastic collocation methods have been identified as effective strategies to compute PCE coefficients \cite{A.Narayan15_SC}. Stochastic collocation allows one to treat existing deterministic simulation models as black box routines in a larger pipeline for performing parametric analysis with PCE. Popular stochastic colloation approaches include sparse grids approximations \cite{Agarwal_2009domainadasc, Bieri_2011SCC, Eldred,GZ, Mazabaras, F.Nobile08_s}, pseudospectral projections \cite{reagan}, and least squares
approaches \cite{J.hampton2015Cm,Tang_2014DLSp,Zhou_Narayan_Xu, Chkifa_2015dlsp, Narayan_2016Christoffel,ZNX,Guo_2018}. Each of these methods requires repeated queries of the black-box simulation model.

In many practical applications, scarce computational resources limit the number of possible queries for the black-box simulatino model, thus limiting the amount of available information about the function $f$, and this makes accurate approximation of the PCE coefficients a difficult task. One popular computational strategy that constructs PCE approximations with limited information is stochastic colloation via $\ell_1$-minimization  \cite{Doostan11_nonada,Yan12_Sc,pengdoostan,J.Hampton2015Cs,JNZ,Guo_2016scrqgauss}. The approach is very effective when the number of non-zero terms in the PCE approximation of the model output is small (i.e. $f$ has a sparse represenation in the PCE basis) or the magnitude of the PCE coefficients decays rapidly (i.e. the PCE expanstion of $f$ has a compressible representation).

In this paper, we consider a gradient enhanced  $\ell_1$-minimization approach for constructing PCE ceofficients. We consider $\ell_1$ minimization with both function and gradient evaluations. Recent advances \cite{Roderick10_Pr,Alekseev,Li11_Oba,Baar,Lockwood13_Gb,jakeman,Peng_2016gradient} have shown that the inclusion of derivative evaluations have the potential to greatly enhance the construction of surrogates especially if those derivatives can be obtained inexpensively, e.g. by solving adjoint equations \cite{Griewank_2003ad}. Potential applications of this approach also include Hermite-type interpolative approximations \cite{XUZHOU,Ben,Ward}. The gradient enhanced approach here can be viewed as a Hermite-type interpolation, however, such an approach differs from classical Hermite interpolation (see e.g., \cite{Interpolation_1,Interpolation_2,Interpolation_3,Learning_1}), since this approach seeks to finding a \textit{sparse} representation.

The main contribution of this work is to present a general framework to include the gradient evaluations in an $\ell^1$ minimization framework. More precisely, we design appropriate preconditioners for the measure matrix, and we show that the inclusion of these derivative measurements can almost-surely lead to improved conditions for a successful solution recovery. The framework is quite general, and it applies to approximation of functions with either bounded or unbounded domain. Comparisons between the gradient-enhanced approach and standard $\ell_1$-minimization are also presented, and numerical examples suggest that the inclusion of derivative information can guarantee sparse recovery at a reduced computational cost.

The rest of the paper is organized as follows. In section 2, we present some preliminaries for the collocation methods with $\ell_1$ minimization, we call this the ``standard" approach.  The gradient-enhanced $\ell_1$ minimization approach is presented in Section 3, and this is followed by some further discussions in Section 4. Numerical examples are provided in Section 5, and we finally give some conclusions in Section 6.

\section{Preliminaries}
\subsection{Generalized polynomial chaos expansions}
Let $\bx=(x_1 , \ldots , x_d)^\top$ be a random vector with $d$ mutually independent components; each $x_i$ takes values in $\Gamma^i \subset \mathbb{R}.$ Since the variables $\{x_i\}_{i=1}^d$ are mutually independent, their marginal probability density functions $\rho_i$, associated with random variable $x_i$, completely characterize the distribution of $\bx$. Define $\Gamma:= \otimes_{i=1}^d\Gamma^i \subset \mathbb{R}^d,$ and let $\rho(\bx)= \prod_{i=1}^d \rho_i(x_i): \Gamma\rightarrow \mathbb{R}^+$ denote the joint probability density function of $\bs{x}.$

Our objective is to approximate the QoI $f(\bx):\Gamma \to \mathbb{R}$. In a simple stochastic collocation approach, we wish to recover information about this function from limited set of function evaluations. In this paper, we seek this approximation using a PCE and so we first introduce the multivariate orthogonal PCE basis.

For each marginal density $\rho_i,$  we can define the univariate PCE basis elements, $\varphi^i_n$, which are polynomials of degree $n$, via the orthogonality relation
\begin{align}\label{eq:phi-orthonormality}
  \mathbb{E} \left[\varphi^i_n(x_i) \varphi^i_{\ell}(x_i)\right] = \int_{\Gamma^i} \varphi^i_n(x_i) \varphi^i_{\ell}(x_i) \rho_i(x_i) dx_i = \delta_{n,\ell}, \quad n,\,\,\ell \geq 0,
\end{align}
with $\delta_{n,\ell}$ the Kronecker delta function. Up to a multiplicative sign, this defines the polynomials $\varphi^i_n$ uniquely; thus the probability measure $\rho_i$ determines the type of orthogonal polynomial basis. For example, the Gaussian (normal) distribution yields the Hermite polynomials, the uniform distribution pairs with Legendre polynomials, etc. For a more detailed account of the correspondence, see \cite{Xiu2002_WPC}.

One convenient representation for a multivariate gPC basis is as a product of the univariate gPC polynomials in each direction. We define
\begin{align}
\label{gpcbais}
  \psi_{\bsn}(\bx) \coloneqq \prod_{i=1}^d \varphi^{i}_{n_i}\left(x_{i}\right),
\end{align}
where $\bsn = \left(n_1, \ldots, n_d\right) \in \N_0^d$ is a multi-index set with $|\bsn| = \sum_{i=1}^d n_i$.  The product functions $\psi_{\bsn}$ are $L^2$ orthogonal under the joint probability density function $\rho$ for $\bx$:
\begin{align}
\label{othorgonal}
\mathbb{E} \left[\psi_{\bsn}(\bx) \psi_{\bsj}(\bx)\right] &= \int_{\Gamma} \psi_{\bsn}(\bx) \psi_{\bsj}(\bx) \rho(\bx) d\bx = \delta_{\bsn,\bsj}, & \bs{n}, \bs{j} &\in \N_0^d
\end{align}
where $\delta_{\bsn, \bsj} = \prod_{i=1}^d \delta_{n_i, j_i}$.

We denote by $T_n^d$ the total degree space, i.e., the space of $d$-variate algebraic polynomials of degree $n$ or less. An element $f_n$ in $T_n^d$ has a unique expansion in the $\phi_{\bs{n}}$ basis:
\begin{align}
\label{eq:pce1}
f_n=\sum_{\bsk \in \Lambda^T_{n}}c_{\bsk}\psi_{\bsk}(\bx),
\end{align}
where $ \Lambda^T_{n}$ is the total-degree multi-index set,
\begin{align*}
  \Lambda^T_{n} &\coloneqq \left\{ \bsk \in \mathbb{N}_0^d\,\, \big| \,\, \sum_{i=1}^d k_i \leq n \right\}
\end{align*}
The dimension of $T_n^d$ is
\begin{align}\label{eq:td-dim}
M=  \left|\Lambda^T_{n}\right| \coloneqq \dim T_n^d = \left( \begin{array}{c} d+n \\ n \end{array}\right).
\end{align}
By defining an(y) total order on the elements of $\Lambda^T_n$, we can re-write \eqref{eq:pce1} as the following scalar-indexed version
\begin{align}
\label{eq:pce}
f_n=\sum_{\bsk \in \Lambda^T_{n}}c_{\bsk}\psi_{\bsk}(\bx)=\sum_{j=1}^{M}{c}_j \psi_j(\bx),
\end{align}
where $\bs{c} \in \R^M$ contains the vector of expansion coefficients, and hence uniquely defines a function $f_n$.

%A polynomial chaos expansion is defined as $s$-sparse when $\|\bs{c}\|_0\leq s$, i.e the number of non-zero coefficients, does not exceed $s$. The aim of this paper is to estimate the projected function $f_n$ of $f$ in the total degree space $T_n^d$, using the information
%at $N (typically, N\ll M)$ points. In particular, we shall also assume that, beside the function values, the % gradient
%information on these samples is also available. More precisely, we denote the sampling set as
%\begin{align*}
%\Xi :=\{\mathbf{z}_1,...,\mathbf{z}_N\}
%\end{align*}
%One of our main tasks is to reconstruct $f_n\in T_n^d$ from its sample evaluations and the gradient
%evaluations:
%\begin{align*}
%&y=f(\mathbf{z}),\quad \mathbf{z}\in \Xi,\\
%&\partial_k(y)=\partial_k f(\mathbf{z}), \quad \mathbf{z}\in \Xi, \quad k=1,...d,
%\end{align*}
%where $\partial_kf(x) = \frac{\partial f(x)}{\partial x_k}$ stands for the derivative with respect to the kth variable $x_k$.

 \subsection{The compressed sensing approach.} In recent years, stochastic collocation  via compressive sensing is one of the popular approaches to determine the
 coefficients $c_j$ in (\ref{eq:pce}). Such as approach uses fewer evaluations, and seeks to compute a PCE approximation with a sparse coefficient vector.  We denote by $\Xi \subset \Gamma$ a set of samples, i.e.,
\begin{align*}
%\Xi :=\{\bs{z}_1,...,\bs{z}_N\} \subset \Gamma.
  \Xi :=\{\bs{z}^{(1)},...,\bs{z}^{(N)}\} \subset \Gamma.
\end{align*}
We will eventually take $\Xi$ as a collection of $N$ iid samples of a random variable. The standard compressed sensing approach attempts the $\ell_0$ approach,
 \begin{align}
\label{eq:l0minimization}
\argmin_{\mathbf{c}\in \mathbb{R}^M} \|\mathbf{c}\|_0 \quad \text{subject to} \quad \mathbf{\Phi} \mathbf{c} = \mathbf{f},
\end{align}
where $\mathbf{f}=(f(\bs{z}^{(1)}),...,f(\bs{z}^{(N)}))^T$, $\mathbf{c}=(c_1,\dots, c_M)^T \in \mathbb{R}^M$ is the unknown coefficient vector to be determined that defines the PCE expansion \eqref{eq:pce}, and $\mathbf{\Phi} \in \mathbb{R}^{N \times M}$ is the measurement matrix whose entries are
 \begin{eqnarray}\label{matrixelem}
   [\mathbf{\Phi}]_{ij}=\psi_j(\bs{z}^{(i)}), \quad i=1,\dots, N, \quad  j=1,\dots, M.
 \end{eqnarray}
 The $\ell_0$ norm $\|\bs{c}\|_0$ is the number of nonzero entries (the ``sparsity") of the vector $\bs{c}$. The convex relaxation of the above problem is the following $\ell_1$ approach
\begin{align}
\label{eq:l1minimization}
\argmin_{\mathbf{c}\in \mathbb{R}^M} \|\mathbf{c}\|_1 \quad \text{subject to} \quad \mathbf{\Phi} \mathbf{c} = \mathbf{f},
\end{align}
where $\|\bs{c}\|_1$ is the standard $\ell_1$ norm on finite-dimensional vectors. The interpolation condition $\mathbf{\Phi c=f}$ can be relaxed to $\mathbf{\|\Phi c-f\|_2}\leq \epsilon$, for some tolerance value $\epsilon$ and with $\|\cdot\|_2$ the vector Euclidean norm, resulting in a regression type "denoising" approach.

Fixing $M$, certain conditions on $N$ and $\bs{\Phi}$ can guarantee that the $\ell_1$-relaxed minimization \eqref{eq:l1minimization} produces the sought solution to the $\ell_0$ problem \eqref{eq:l0minimization}. Several types of such sufficient conditions on $\vPhi$ have been presented in the compressive sampling (CS) literature, such as the mutual incoherence property (MIP) and restricted isometry property (RIP). Our invesigation in this paper concerns the MIP: The mutual incoherence constant (MIC) of $\mathbf{\Phi}$ is defined as
\begin{equation}\label{eq:MIC}
\mu\,\,=\,\,\mu(\mathbf{\Phi})\,\,:=\,\, \max_{k\neq j}\frac{\abs{\innerp{\mathbf{\Phi}_k, \mathbf{\Phi}_j}}}{\|\mathbf{\Phi}_k\|_2 \cdot \|\mathbf{\Phi}_j\|_2},
\end{equation}
where $\bs{\Phi}_j$ is the $j$th column of $\bs{\Phi}$. Assume that $\mathbf{c}_0$ is an $s$-sparse vector in $\mathbb{C}^M$, i.e.,  $\|\vc\|_0\leq s,$ and if
\begin{equation}\label{eq:BPMIC}
\mu \,\,<\,\,\frac{1}{2s-1},
\end{equation}
then the solution to the $\ell_1$ minimization (\ref{eq:l1minimization}) with $\mathbf{f}=\mathbf{\Phi}\mathbf{c}_0$ is exactly $\mathbf{c}_0$, i.e.,
$$
 \mathbf{c}_0=\argmin_{\mathbf{c}\in \C^M} \left\{\|\mathbf{c}\|_1  \,\, \text{\rm subject to}\,\, \mathbf{\Phi}\mathbf{c}=\mathbf{\Phi}\mathbf{c}_0\right\}.
$$
This result was first presented in \cite{DonHuo} for the case with $\mathbf{\Phi}$ being the union of two orthogonal matrices, and was later extended to general matrices by Fuchs \cite{Fuchs} and Gribonval \& Nielsen \cite{GrNi}.  In \cite{Cai_Wang_Xu}, it is also shown that
 $\mu < \frac{1}{2s-1}$ is sufficient for stable approximation of $\mathbf{c}$ in the noisy case.

\section{A gradient enhanced compressed sensing approach}
We consider inclusion of gradient measurements in an $\ell_1$ optimization approach for compressed sensing. The motivation is that the gradient measurements can usually be obtained in a relatively inexpensive way from model simulations, e.g, by using the adjoint techniques \cite{Griewank_2003ad}. Consider the availability of the following data:%More precisely, we shall consider the $\ell_1$ approach with the following information
\begin{align*}
&y=f(\bs{z}),  \qquad \qquad \bs{z}\in \Xi,\\
&\partial_k(y)=\partial_k f(\bs{z}), \quad \,\,\bs{z}\in \Xi, \quad k=1,...d,
\end{align*}
where $\partial_kf(\bx) = \frac{\partial f(\bx)}{\partial x_k}$ stands for the derivative with respect to the $k$th variable $x_k$.

Then concatenating all the measurement conditions above into matrix-vector format in an $\ell_1$ optimization problem yields the following approach:
\begin{align}\label{eq:gradientl1}
\argmin_{\mathbf{c}\in \mathbb{R}^M}\|\mathbf{c}\|_1 \quad \textmd{subject} \ \textmd{to} \quad   \mathbf{W}\tilde{\mathbf{\Phi}}\mathbf{P}\mathbf{c}=\mathbf{W}\tilde{\mathbf{f}}
\end{align}
with
\begin{align*}
\tilde{\mathbf{f}}=\left(\begin{array}{l}
\mathbf{f} \\
\mathbf{f}_{\partial}
\end{array}\right), \quad
\tilde{\mathbf{\Phi}}=\left(\begin{array}{l}
\mathbf{\Phi} \\
\mathbf{\Phi}_{\partial}
\end{array}\right),\quad
\mathbf{\Phi}_{\partial}=& \left[ \begin{array}{c}
\frac{\partial\mathbf{\Phi}}{\partial x_1} \\
\vdots \\
 \frac{\partial\mathbf{\Phi}}{\partial x_d}\\\end{array} \right ], \quad \mathbf{f}_{\partial}=\left[ \begin{array}{c}
\frac{\partial\mathbf{f}}{\partial x_1} \\
\vdots \\
 \frac{\partial\mathbf{f}}{\partial x_d}\\\end{array} \right ]
\end{align*}
where for $k=1,...,d$, $\frac{\partial\mathbf{\Phi}}{\partial x_k}\in R^{N\times M}$, $\frac{\partial\mathbf{f}}{\partial x_k}\in R^N$ are defined as following
\begin{align*}
\bigg [\frac{\partial\mathbf{\Phi}}{\partial x_k}\bigg ]_{ij}=\frac{\partial\psi_{j}(\bx)}{\partial x_k}(\bs{z}_i),\quad \bigg [\frac{\partial\mathbf{f}}{\partial x_k}\bigg ]_{i}=\frac{\partial f(\mathbf{\bx})}{\partial x_k}(\bs{z}_i), \quad i=1,...,N, \quad  j=1,...,M.
\end{align*}
Note that now $\tilde{\mathbf{\Phi}}\in \mathbb{R}^{N(d+1)\times M}$, and we refer to this matrix as the gradient-enhanced measurement/design matrix,
with $\tilde{\mathbf{f}}\in \mathbb{R}^{N(d+1)}$ is the data vector.

Notice that compared to the standard $\ell_1$ approach, the gradient enhanced approach (\ref{eq:gradientl1}) involves two additional matrices:
\begin{itemize}
\item  The preconditioning matrix $\mathbf{W}:$ this is designed to enhance recovery properties in $\ell_1$ optimization. Its definition will depend on the type of PCE basis and on how the sample set $\Xi$ is generated. We will discuss this in detail later.
\item The normalizing/weighting matrix  $\mathbf{P}:$ this matrix is included to normalize the design matrix, so that $\widehat{\mathbf{\Phi}}:=\mathbf{W}\tilde{\mathbf{\Phi}}\mathbf{P}$ satisfies mean isotropy.
\end{itemize}
We shall show that the preconditioned matrix $\widehat{\mathbf{\Phi}}:=\mathbf{W}\tilde{\mathbf{\Phi}}\mathbf{P}$  is much more stable in the sense that its MIP (or RIP) constant better behaved than that of the matrix $\tilde{\mathbf{\Phi}}.$  In what follows, we shall give a general guide for choosing these preconditioning matrices.

\subsection{Legendre expansion with Chebyshev sampling}
To illustrate the idea, we begin with Legendre expansion with Chebyshev sampling. I.e., $\rho$ is the uniform measure on $\Gamma = [-1,1]^d$, the PCE basis functions $\psi_j$ are tensor-product Legendre polynomials, and $\Xi$ is constructed via iid sampling from the Chebyshev (arcsine) measure. The use of Chebyshev sampling when approximating with a Legendre polynomial basis (where available data is only function values) has been widely investigated \cite{Rauhutward,XUZHOU,JNZ}, and can produce better results (compared to uniform sampling) when large-degree approximations are required. Here we shall show how inclusion of gradient information can be accomplished in a systematic way. %More precisely, we assume that the input random variables are distributed according to the uniform measure, and the corresponding expansion is the Legendre expansion.  The random samples $\Xi$ are chosen with Chebyshev sampling strategy.

Suppose that $\Xi$ is comprised of $N$ iid samples generated from the uniform measure $\rho$. Since the (orthonormal) Legendre polynomials satisfy \eqref{othorgonal}, then we have
\begin{align}
\mathbb{E}\left[\frac{1}{N} \mathbf{\Phi}^T \mathbf{\Phi}\right] = \mathbf{I}.
\end{align}
This is the mean isotropy property. However, if we instead construct $\Xi$ as $N$ iid samples from a different measure, say the Chebyshev measure, then we must introduce a preconditioner to retain the mean isotropy property. Our gradient-enchanced $\ell_1$ minimization strategy aims to maintain mean isotropy when gradient evaluations are included in the measurement matrix.

We recall a standard fact, that derivatives of the univariate Legendre polynomials are orthogonal with respect to the weight function $\eta(x)= (1-x^2)$ \cite{Szego_1959}. By using the above facts we can derive that if $\bz \in \R^M$ is a random variable distributed according to the product Chebyshev weight function,
\begin{align*}
  \rho_c(\bs{x}) = \prod_{j=1}^d \frac{1}{\pi \sqrt{1 - x_j^2}},
\end{align*}
then we have
\begin{align}
  \mathbb{E}^c\left[\frac{2^d}{\rho_c(\bz)} \psi_i(\bz) \psi_j(\bz) +\sum\limits_{k=1}^d \frac{1-z_k^2}{\rho_c(\bs{z})}\frac{\partial \psi_i}{\partial x_k}(\bz)\frac{\partial \psi_j}{\partial x_k}(\bz)\right]=\delta_{ij}\bigg (1+\sum\limits_{k=1}^d c_k i_k(i_k+1)\bigg), \label{eq:legendre_chebyshev}
\end{align}
where $c_k$ is a constant that we make precise later. Here we use $\mathbb{E}^c$ to emphasize that the expectation is taken with respect to the Chebyshev measure.%, and $\rho_c$ is Chebyshev weight function $\rho_c = \frac{1}{\pi \sqrt{1-x^2}}.$

The above derivation suggests the following choices for the matrices $\mathbf{W}$ and $\mathbf{P}:$
\begin{align}
\mathbf{W}=\left[ \begin{array}{cccc}
\mathbf{W}^0 & & &\\
&\mathbf{W}^1 & &\\
 & & \ddots &\\
&&& \mathbf{W}^d\\\end{array}\right],
\end{align}
where $\mathbf{W}^k$ are diagonal matrices whose entries are defined as
\[
  \mathbf{W}_{n,n}^0=\bigg ((4/\pi^2) (1-(z^{(n)}_j)^2)\bigg )^{d/4},  \quad
  \mathbf{W}^j_{n,n}= \frac{\mathbf{W}_{n,n}^0}{\sqrt{2}} \left(1 - \left(z_j^{(n)}\right)^2\right)^{1/2}, \quad j=1,...,d, \quad n=1,...,N.
\]
Here $z^{(n)}_j$ is the $j$th component of the random vector $\mathbf{z}^{(n)}.$  The normalizing matrix $\mathbf{P}$ is a diagonal matrix with entries $\mathbf{P}_{i,i}=\left(1+\sum\limits_{k=1}^d c_k i_k(i_{k}+1)\right)^{-1/2}$.

With the above definitions, one can easily show that the design matrix is mean isotropy, namely,
\begin{align}
\mathbb{E}^c\left[\frac{1}{N}\widehat{\mathbf{\Phi}}^T \widehat{\mathbf{\Phi}}\right] = \mathbf{I}, \quad \textmd{with}  \quad \widehat{\mathbf{\Phi}}=\mathbf{W}\tilde{\mathbf{\Phi}}\mathbf{P}.
\end{align}
This is the general strategy for our gradient formulation: we take the sampling measure from which $\Xi$ is constructed to be a degree-asymptotica ``good" sampling measure for the PCE basis $\psi_j(\bx)$, we design a preconditioning matrix so that the PCE basis is mean isotropic, and finally we choose a weighting matrix $\bs{P}$ to retain isotropy of the gradient evaluations. Having shown the idea for the special case of Legendre polynomials, we now generalize to arbitrary Jacobi families.

\subsection{General Jacobi expansions with Chebyshev sampling}\label{sec:jacobi}
Now, we turn to the case of General Jacobi expansions with Chebyshev sampling, which includes the Legendre expansion with uniform sampling as a special case. The univariate probability density
\begin{align}\label{eq:jacobi-density}
\rho^{(\alpha, \beta)}(x) &= d^{(\alpha,\beta)}(1-x)^\alpha (1+x)^\beta, & \alpha,\beta &\geq - \frac{1}{2}
\end{align}
is the Beta density function on $[-1,1]$. The normalization coefficient is
\begin{align*}
  d^{(\alpha,\beta)}  =\frac{\Gamma(\alpha + \beta + 2)}{\Gamma(\beta+1) \Gamma (\alpha + 1)2^{\alpha+\beta+1} }.
\end{align*}
Keeping with earlier notation, we use $\rho_c \equiv \rho^{(-1/2, -1/2)}$. Then given
\begin{align*}
  \bs{\alpha} &= \left(\alpha_1, \ldots, \alpha_d\right) \in \left[-\frac{1}{2}, \infty\right)^d, &
  \bs{\beta} &= \left(\beta_1, \ldots, \beta_d\right) \in \left[-\frac{1}{2}, \infty\right)^d,
\end{align*}
we can define the notation for multi-dimensional Jacobi probability densities:
\begin{align*}
  \rho^{(\bs{\alpha}, \bs{\beta})}(\bs{x}) &= \prod_{j=1}^d \rho^{(\alpha_j, \beta_j)}(x_j) \\
  %\rho_c(\bs{x}) &= \prod_{j=1}^d \rho_c (x_j) \\
\end{align*}
The multivariate PCE basis elements $\psi$ associated to $\rho^{({\bs{\alpha}, \bs{\beta}})}$ is likewise now well-defined, but to avoid notational clutter we will omit showing explicit dependence of $\psi$ and the measurement matrix $\mathbf{\Phi}$ on $\bs{\alpha}$ and $\bs{\beta}$. By using the identity between Jacobi polynomials and their derivatives, we can derive that if $\bs{z}$ is a random variable distributed according to the measure $\rho_c$, then
\begin{align}\label{eq:jacobi_chebyshev}
  \mathbb{E}\left[ \frac{\rho^{(\bs{\alpha},\bs{\beta})}(\bs{z})}{\rho_c(\bs{z})} \psi_i(\bz) \psi_j(\bz) +
  \sum_{k=1}^d \frac{\rho^{(\bs{\alpha}+\bs{e}_k,\bs{\beta}+ \bs{e}_k)}(\bs{z})}{\rho_c(\bs{z})} \frac{\partial \psi_i}{\partial x_k}(\bz)\frac{\partial \psi_j}{\partial x_k}(\bz)\right]=\delta_{ij} \left (1+\sum\limits_{k=1}^d c^2(i_k, \alpha_k, \beta_k) \right),
\end{align}
where $\bs{e}_j \in \R^d$ is the cardinal unit vector in the $j$th direction; i.e., $(e_j)_k = \delta_{j,k}$. We also define $\bs{e}_0 = \bs{0}$ as the zero vector. The normalization constant $c_k$ is
\begin{align*}
  c^2(i_k,\alpha_k, \beta_k) = i_k (i_k+\alpha_k + \beta_k + 1) \frac{(\alpha_k + \beta_k+2)(\alpha_k + \beta_k + 3)}{4 (\alpha_k+1)(\beta_k+1)}.
\end{align*}
The above derivation suggests the following choices for the matrices $\mathbf{W}$ and $\mathbf{P}:$
\begin{align}
\mathbf{W}=\left[ \begin{array}{cccc}
\mathbf{W}^0 & & &\\
&\mathbf{W}^1 & &\\
 & & \ddots &\\
&&& \mathbf{W}^d\\\end{array}\right],
\end{align}
where $\mathbf{W}^k$ are diagonal matrices whose entries are defined as
\begin{align*}
  W^0_{n,n} = \sqrt{\frac{\rho^{(\bs{\alpha},\bs{\beta})}(\bs{z}^{(n)})}{\rho_c(\bs{z}^{(n)})}},  \quad
  W^j_{n,n} = \sqrt{\frac{\rho^{(\bs{\alpha}+\bs{e}_j,\bs{\beta} + \bs{e}_j)}(\bs{z}^{(n)})}{\rho_c(\bs{z}^{(n)})}}
 %\mathbf{W}^n_{m,m}=\bigg (1-(z^{(n)}_m)^2\bigg )^{1/2}, \quad n=1,...,d, \quad m=1,...,M.
\end{align*}
for $n=1, \ldots, N$, and $j = 1, \ldots d$. The normalizing matrix $\mathbf{P}$ is a diagonal matrix with entries
\begin{align}\label{eq:jacobi-normalizing}
  \mathbf{P}_{i,i}= \left (1+\sum\limits_{k=1}^d c^2(i_k, \alpha_k, \beta_k) \right)^{-1/2}.
\end{align}
%where $ c^2(i_k,\alpha_k, \beta_k) = i_k (i_k+\alpha_k + \beta_k + 1) \frac{(\alpha_k + \beta_k+2)(\alpha_k + \beta_k + 3)}{4 (\alpha_k+1)(\beta_k+1)}$.

With the above definitions, one can, just as for the Legendre case, show that the whole design matrix is mean isotropy, i.e.,
\begin{align}
\mathbb{E}\left[\frac{1}{N}\widehat{\mathbf{\Phi}}^T \widehat{\mathbf{\Phi}}\right] = \mathbf{I}, \quad \textmd{with}  \quad \widehat{\mathbf{\Phi}}=\mathbf{W}\tilde{\mathbf{\Phi}}\mathbf{P}.
\end{align}
%We will prove this mean isotropy in Appendix A.

For this gradient enhanced approach, we are interested in understanding inclusion of derivative information can improve the recovery ability. We shall provide one answer to this question in the following theorem by analyzing the coherence parameter of the design matrix. To this end, we define the coherence parameter of the original compressed sensing approach as
\begin{equation*}
\mu_L(\mathbf{\Phi}) := \sup_{i, \,\mathbf{z}\in \Xi} |\mathbf{\Phi}_i(\bs{z})|^2_2.
\end{equation*}
We have that $|\mathbf{\Phi}_i(\bs{z})|_2$ is the norm of one column in the design matrix $\mathbf{\Phi}$. The parameter $\mu_L$ provides a quantitative recovery quality metric for compressed sensing approaches \cite{E.J.candes2010Ap,J.Hampton2015Cs}. Smaller parameter values result in better recovery properties.

Similarly, following the notation in \cite{Peng_2016gradient}, we define the corresponding parameter of the gradient enhanced approach as
\begin{equation*}
\beta_L\big(\mathbf{\widehat{\Phi}}\big) := \sup_{i, \,\mathbf{z}\in \Xi} \left\| \bs{\widehat{\Phi}}_{i}(\bs{z}) \right\|^2,
\end{equation*}
where
\begin{align*}
  \bs{\widehat{\Phi}}_{i}(\bs{z}) = \frac{1}{\mathbf{P}_{i,i}} \left( \begin{array}{c} \sqrt{\frac{\rho^{(\bs{\alpha},\bs{\beta})}(\bs{z})}{\rho_c(\bs{z})}} \Phi_{i}(\bs{z}) \\
  \sqrt{\frac{\rho^{(\bs{\alpha}+\bs{e}_1,\bs{\beta}+\bs{e}_1)}(\bs{z})}{\rho_c(\bs{z})}} \ppx{x_1} \Phi_{i}(\bs{z}) \\
  \sqrt{\frac{\rho^{(\bs{\alpha}+\bs{e}_2,\bs{\beta}+\bs{e}_2)}(\bs{z})}{\rho_c(\bs{z})}} \ppx{x_2} \Phi_{i}(\bs{z}) \\
  \cdots \\
  \sqrt{\frac{\rho^{(\bs{\alpha}+\bs{e}_d,\bs{\beta}+\bs{e}_d)}(\bs{z})}{\rho_c(\bs{z})}} \ppx{x_d} \Phi_{i}(\bs{z})
  \end{array}\right).
\end{align*}

In the following, we present the main theorem of this paper, which shows the bound for the coherence parameters $\mu_L$ and $\beta_L$.
\begin{theorem}\label{th:main}
Recall that $\mathbf{\Phi}$ and $\widehat{\mathbf{\Phi}}$ are design matrices for the standard $\ell_1$ and the gradient enhanced $\ell_1$ approach via Jacobi expansions with Chebyshev sampling, respectively. Then the two coherence parameters satisfy the following estimates:
\begin{align}\label{eq:coherence-bound}
\mu_L\left(\bs{\Phi}\right) &\leq \prod_{j=1}^d 2 e \left( 2 + \sqrt{\alpha_j^2 + \beta_j^2} \right) \\
\label{eq:d-coherence-bound}
\beta_L\left(\bs{\widehat{\Phi}}\right) &\leq C\prod_{j=1}^d 2 e \left( 2 + \sqrt{\alpha_j^2 + \beta_j^2} \right)
\end{align}
where $1 \leq C \leq 1 + \frac{\sqrt{2}}{2}$. The lower bound for $C$ is achieved when $\alpha_k = \beta_k = -\frac{1}{2}$ and the upper bound occurs when there is a $k$ such that $\alpha_k = \beta_k = 0$. If $\mathcal{N}(\cdot)$ represents the nullspace of a matrix, then
$$\mathcal{N}\big(\mathbf{\widehat{\Phi}}\big) \subset \mathcal{N}\big(\mathbf{\Phi}\big),$$
and this is almost-surely a strict subset when $\mathbf{\Phi}$ is under-sampled.
\end{theorem}
\begin{proof}
See Appendix A2.
\end{proof}
We remark that ideally we could show the gradient approach admits an improved (smaller) parameter $\beta_L$, i.e. $\beta_L\left( \bs{\widehat{\Phi}} \right) \leq \mu_L\left( \bs{\Phi} \right)$, yielding a better recovery property. Our analysis does not bear this fruit, but we have shown that (i) the coherence for both $\bs{\Phi}$ and $\bs{\widehat{\Phi}}$ is a constant raised to the $d$th power, independent of polynomial degree; and (ii) the constant $C$ in the estimate \eqref{eq:d-coherence-bound} is dimension-independent and relatively small.

\subsection{Hermite expansions with Gaussian sampling}
In the last two sections, we presented two examples in bounded domain. Here we present a unbounded case, where the basis elements are Hermite polynomials and the samples are chosen according to the Gaussian measure. The authors in \cite{Peng_2016gradient} notice that the gradient of the Hermite basis elements are orthogonal with respect to the same Gaussian measure. The authors show that if $\psi_j$ are suitably normalized Hermite polynomials and $\bs{z}$ is a multivariate standard normal random variable, then
\begin{align}
\label{hermite-l2norm}
\mathbb{E}\bigg (\psi_i(\bz)\psi_j(\bz)+\sum\limits_{k=1}^d\frac{\partial \psi_i}{\partial x_k}(\bz)\frac{\partial \psi_j}{\partial x_k}(\bz)\bigg)=\delta_{ij}\bigg (1+\sum\limits_{k=1}^di_k\bigg ).
\end{align}
This motivates the following choice of normalizing matrix:
$$\mathbf{P}^h= \textmd{diag}(\mathbf{P}_{1,1}, ... , \mathbf{P}_{N,N}), \quad \mathbf{P}_{i,i}=\bigg (1+\sum\limits_{k=1}^di_k\bigg )^{-1/2}, \quad i=1, ... , N. $$
The preconditioning matrix $\bs{W}$ would be set to the identity in this case. One main result of \cite{Peng_2016gradient} then shows a similar result as in Theorem \ref{th:main}.

We also remark that extensions to general unbounded problems (e.g., Laguerre expansions) would use similar techniques as above. We note that there are more sophisticated sampling strategies one can use in the unbounded case \cite{Narayan_2016Christoffel,JNZ} so that the choice of $\bs{W} = \bs{I}$ is not necessarily optimal.

Finally, we make some remarks about the weighting matrices $\bs{P}$ that we have constructed. Our choice of this matrix for the Hermite case above, and for the general Jacobi case in \eqref{eq:jacobi-normalizing} have been diagonal matrices due to the orthogonality property of \textit{derivatives} of orthogonal polynomials. In fact, the only univariate polynomial families whose derivatives are also sets orthogonal polynomials are the Jacobi, Laguerre, and Hermite polynomials \cite{hahn_uber_1935,webster_orthogonal_1938,krall_derivatives_1936}. Therefore, if a PCE basis associated to a non-classical polynomial family is used, then the choice of $\bs{P}$ will not be diagonal: instead it will be any inverse square root of the Gramian associated to the polynomial derivatives.

\section{Further discussions}

In the last section, we have present a general framework to include the gradient information in the compressed sensing approach. Notice that in our approach, the gradient information is included directly for each direction (variable). However, one may consider different ways to include those information.  For instance, partial gradient measurements, e.g., an incomplete set of directional derivatives, may be provided. We may therefore consider the following problem:

\begin{itemize}
\item Find a sparse expansion of $f(\bx)$ with
\begin{align}
  f(\bs{z}^{(j)})\,\,&=\,\,f_j, \qquad\qquad\qquad\qquad\,\, \bs{z}^{(j)}\in \Xi, \label{eq:value1}\\
  D_{\vv_t}f(\bs{z}^{(j)})\,\,&=\,\,f'_{j,t},\qquad t=1,\ldots,k, \,\,\,\bs{z}^{(j)}\in \Xi, \label{eq:value2}
\end{align}
where  $D_{\vv_t}f(\bs{z}_j):=\innerp{\nabla f(\vx),\vv_t}|_{\vx=\bs{z}_j}$ and $\vv_t\in \R^d$ are directional vectors. Namely, we assume that both function values and the directional derivative information at the sampling points are known.
\end{itemize}
The above approach can be viewed as a generalization of the approach in the last section. Here, we have more flexibility to choose the directions $\{\vv_j\}_j,$ and it is expected that a smart choice of $\{\vv_j\}_j$ may lead to a improved recovery results. However this approach might not be of practical value, as there is no evidence to show how to get such directional derivatives. Nevertheless, this can be viewed as an interesting mathematical problem, as discussed in \cite{xu_zhou_2018}.

Besides the above approach, one may also interested in the following mathematical problem:
\begin{itemize}
\item Find a sparse approximation of $f(\bx)$ with
\begin{equation}\label{eq:hi}
D_{\vv_j}^{\tau_j}f(\bs{z}_j)\,\,=\,\, y_j, \quad \bs{z}_j\in \Xi, \quad j=1,\ldots,N,
\end{equation}
where $\vv_j\in \R^d$ are directional vectors, and $\tau_j\in \N_0$ are non-negative integers.
\end{itemize}
Here, it is supposed that one knows either the $\tau_j$-order directional derivative of $f$ at $\bs{z}^{(j)}$ or the function value $f(\bs{z}^{(j)})$. If $\tau_j=0$, then (\ref{eq:hi}) means that we know only the function value of $f$ at $\bs{z}^{(j)}$, i.e., $y_j=f(\bs{z}^{(j)})$. Notice that a main feature of this approach is that the locations (samples) for evaluating the function values and the gradient information are independent, while normally one assumes that function values and the gradient information are evaluated in the same locations (which is more practical).

Finally, we would like to remark that for the gradient-enhanced approach it seems that the precondition matrix is the key for the recovery property. We believe that such matrices presented here is not optimal, and one may consider alternative choices, e.g., the Christoffel weighted approach in \cite{Narayan_2016Christoffel,JNZ} that is optimal for degree-asymptotic approximations.

\section{Numerical examples}

We now provide some numerical examples to show the performance of the gradient-enhanced $\ell_1$-minimization approach.  For the implementation of the $\ell_1$ minimization, we employ the available tools such
as Spectral Projected Gradient algorithm (SPGL1) from \cite{Vanden} that was implemented in the MATLAB package SPGL1 \cite{vanFrie2}. To compare the standard and gradient-enhanced $\ell$- minimization solutions, we will use \textit{standard} to denote the numerical results by using the standard $\ell_1$-minimization, while we shall denote by \textit{gradient-enhanced} the numerical results obtained by using gradient enhanced $\ell_1$ approach. We shall also use \textit{standard-double} to denote the standard approach with "doubled" function values.  More precisely, consider for example a two dimensional example, suppose we have $N$ function values and $2N$ gradient values (with respect to each variable). Then, the full gradient enhanced approach will use $3N$ information ($100\%$ information, i.e., $N$ function values and $2N$ gradient values). A $50\%$ gradient enhanced approach would involve $N$ function values and $N$ gradient information (with respect to a randomly chosen direction/variable). While the \textit{standard-double} will stands for the standard approach with $3N$ function values.

\subsection{Stability tests}

We first show some stability tests between the preconditioned matrix $\widehat{\mathbf{\Phi}}=\mathbf{W}\tilde{\mathbf{\Phi}}\mathbf{P}$ and the original matrix $\tilde{\mathbf{\Phi}}.$ This is done by showing the MIP constant in equation (\ref{eq:MIC}),  which is a key index for stable sparse recovery. Notice that the smaller the MIC constant is, the better the recovery guarantee. We consider the Legendre expansion with Chebyshev sampling. For a fixed polynomial space, we show in Fig.\ref{fig:legendre_MIC_via_samples} the MIP constants of $\widehat{\mathbf{\Phi}}$ and $\tilde{\mathbf{\Phi}}$ with respect to the number of samples. While Fig.\ref{fig:legendre_MIC_via_number of PCE} presents the MIP constants, for a fixed number of samples, with respect to the number of expansion terms $M.$ In both cases, we also present the MIP constant of the matrix $\mathbf{\Phi}$ where no derivative information is included. It is clear shown that the preconditioned matrix $\widehat{\mathbf{\Phi}}$ admits a much well behaved MIC constant (see the purple-triangular lines). While it is also shown that the direct inclusion of derivative information (the matrix $\tilde{\mathbf{\Phi}}$) can actually destroy the stability of the matrix $\mathbf{\Phi}$ (see the blue lines and red lines).
\begin{figure}[!ht]
\centering
\includegraphics[width=0.48\textwidth,height=0.24\textheight]{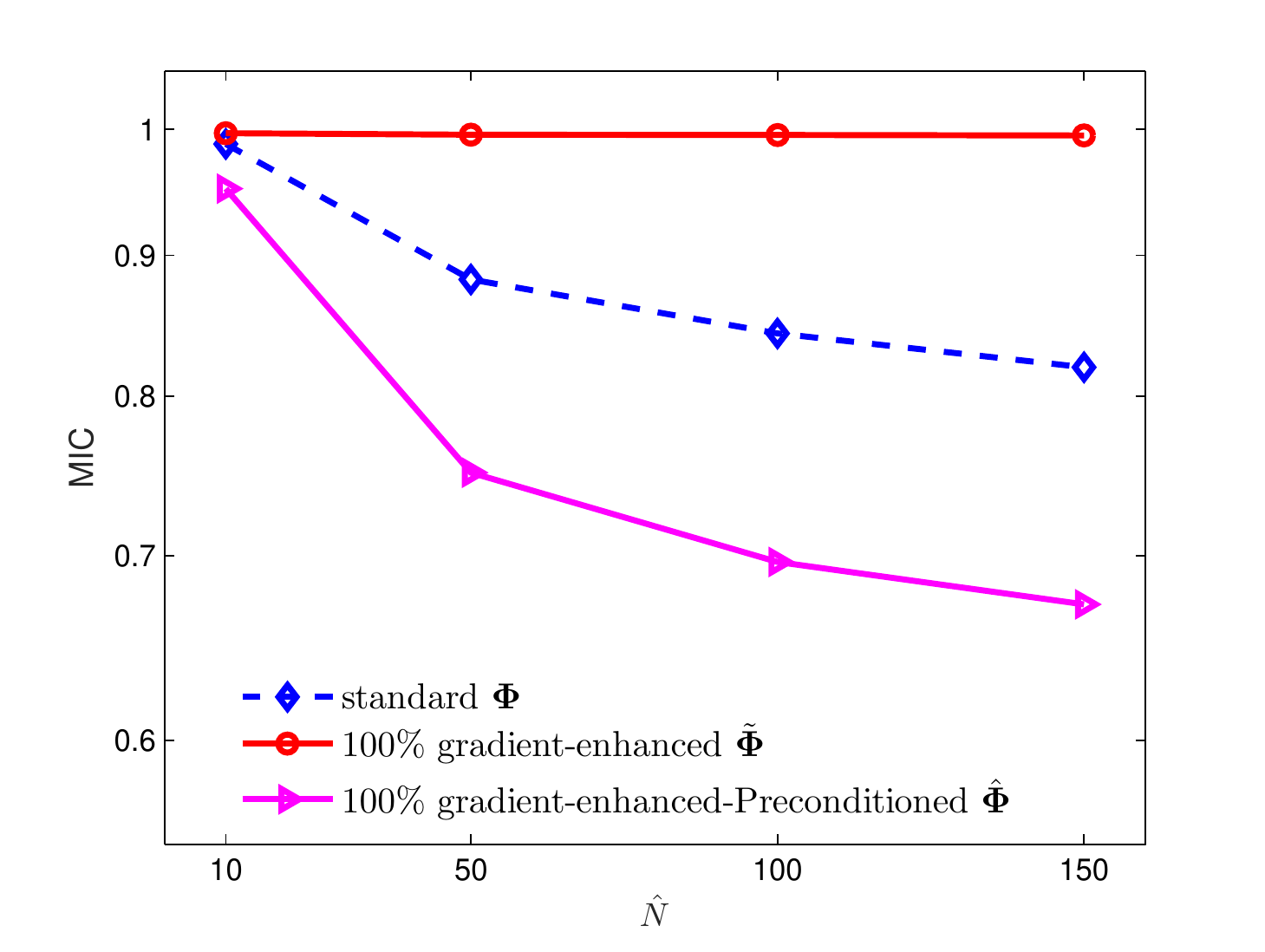}\quad
\includegraphics[width=0.48\textwidth,height=0.24\textheight]{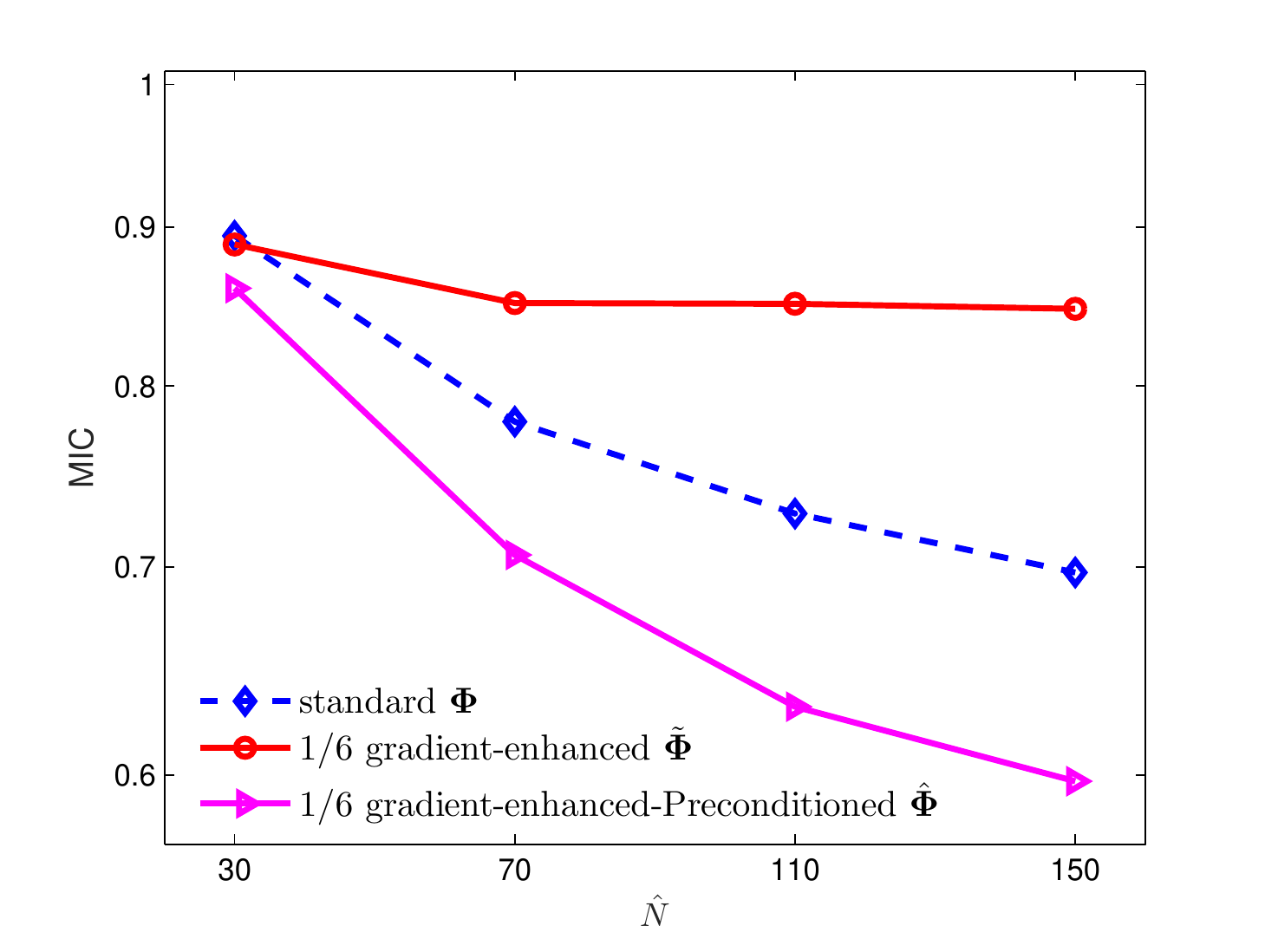}
\caption{The MIP constant for three matrices against the number of samples: $\widehat{\mathbf{\Phi}}$, $\tilde{\mathbf{\Phi}},$ and $ \mathbf{\Phi}$.
Left: $d=2$, $n=30$. Right: $d=6$, $n=5$. Legendre polynomial and Chebyshev samples.}\label{fig:legendre_MIC_via_samples}
\end{figure}

\begin{figure}[!ht]
\centering
\includegraphics[width=0.48\textwidth,height=0.24\textheight]{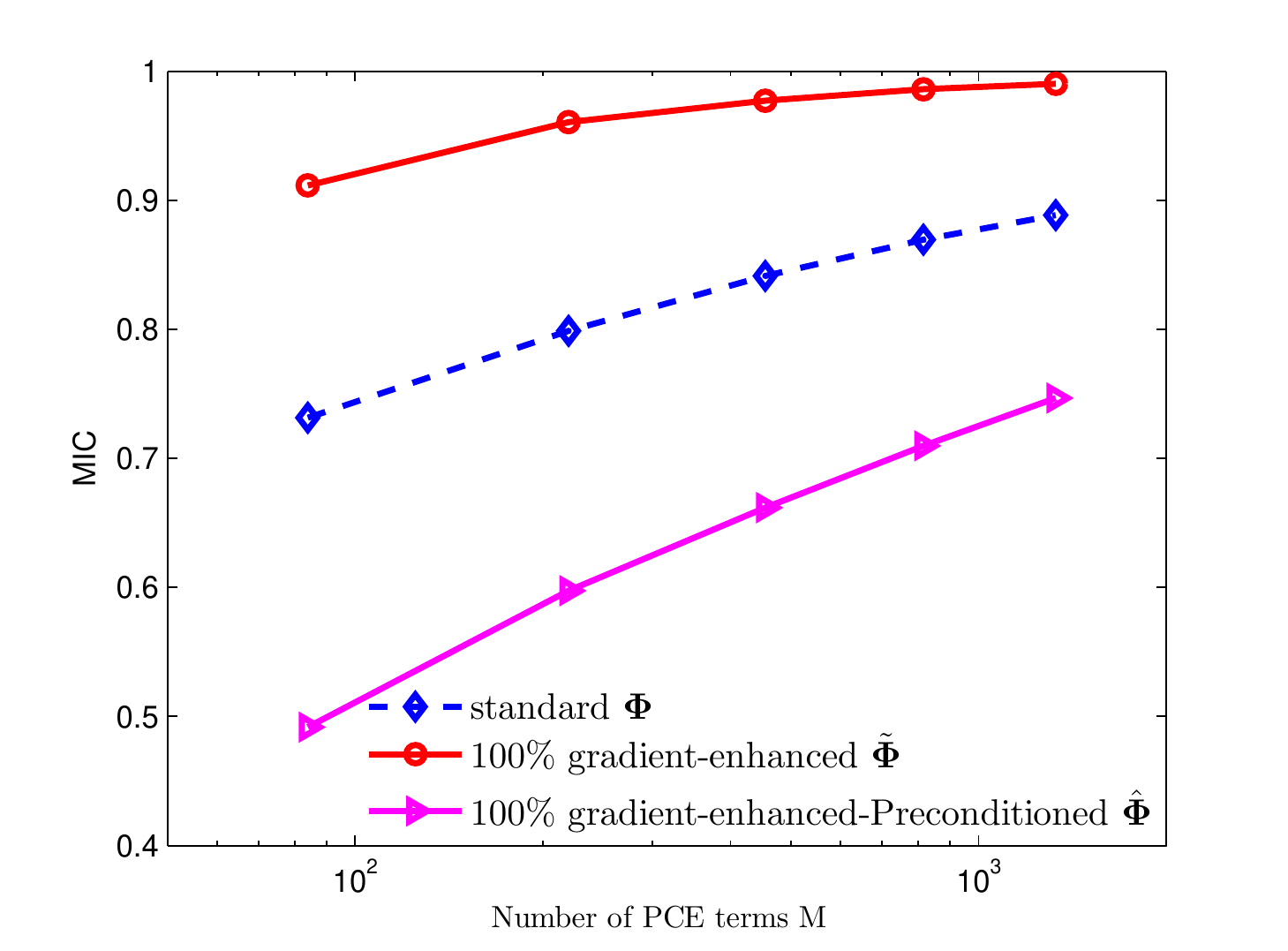}\quad
\includegraphics[width=0.48\textwidth,height=0.24\textheight]{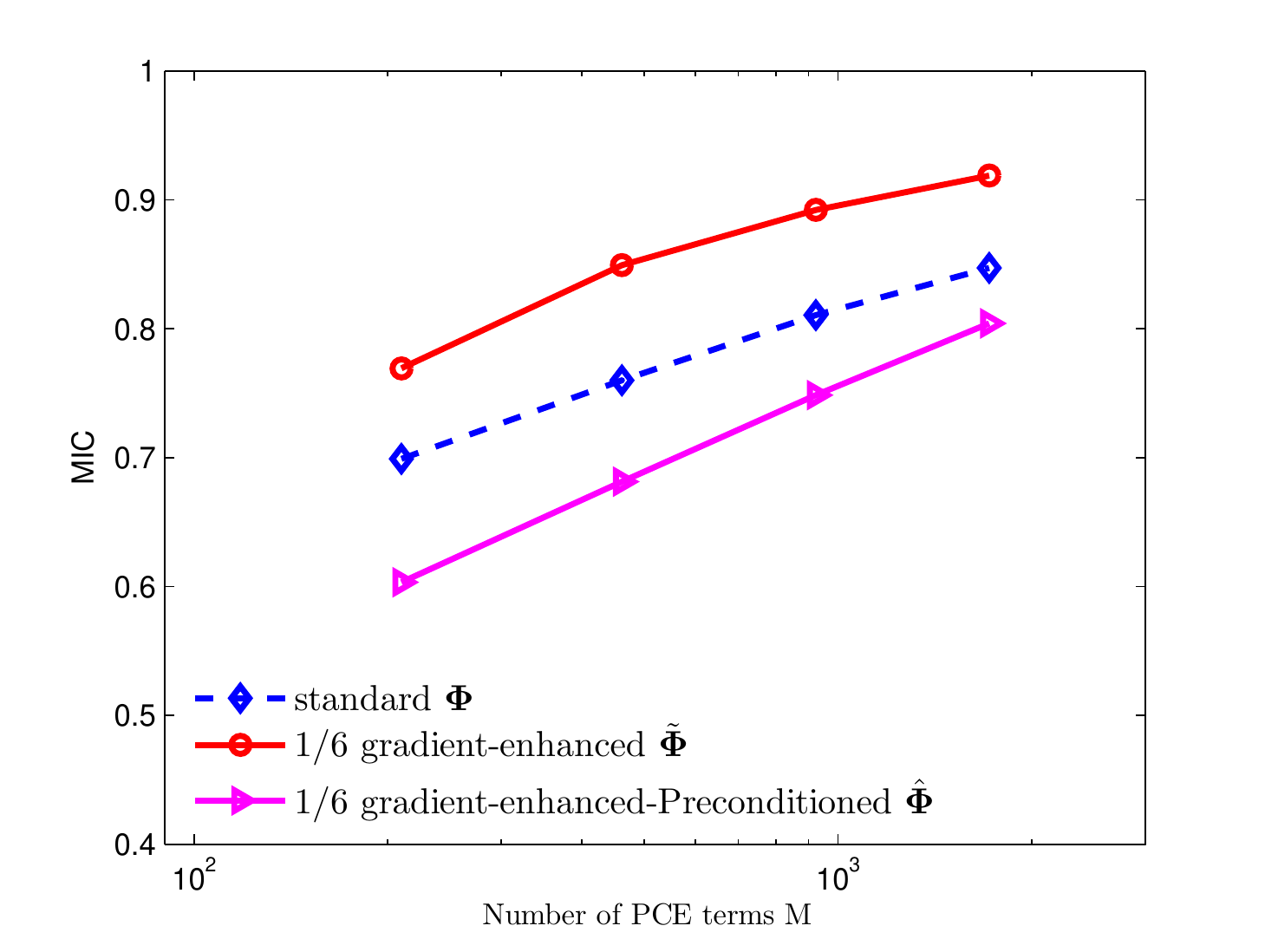}
\caption{The MIP constant for three matrices against the number of PCE terms with fixed number of samples: $\widehat{\mathbf{\Phi}}$, $\tilde{\mathbf{\Phi}},$ and $ \mathbf{\Phi}$.
Left: $d=2$, $N=80$. Right: $d=6$, $N=80$. Legendre polynomial and Chebyshev samples.}\label{fig:legendre_MIC_via_number of PCE}
\end{figure}

\subsection{Benchmark Test: fixed sparsity}
In this section, We assume that the target (exact) function has a sparse polynomial expansion, i.e. $f(\bx)=\sum _{j=1}^M c_j \psi_j(\bx)$ with $\|\mathbf{c}\|_{0}=s$, and attempt to recover this vector. In all our tests, we assume the random input is uniform distributed, and the samples are chosen randomly with the chebyshev measure. Notice that numerical examples for the Hermite expansion can be found in \cite{Peng_2016gradient}.

For a given sparsity level $s$, we shall fix $s$ coefficients of the polynomial while keeping the rest of the coefficients zero. The values of the $s$ non-zero coefficients are drawn as $i.i.d.$ samples from a standard normal distribution. We approximate the PCE coefficients $\mathbf{c}$ via the gradient enhanced approach
from these generated data. We examine the frequency of successful recoveries. This is accomplished by $100$ trials of the algorithms and counting the successful ones. A recovery is considered successful when the resulting coefficient vector $\mathbf{c}$ satisfies $\|\mathbf{c}-\mathbf{\tilde{c}}\|_{\infty}\leq{10^{-3}}.$

We consider the two dimensional case first. In Figure.\ref{fig:legendre2d_gradient} (Left), we show the recovery probability against number of sample points $N$ with a fixed sparsity $s=8$. To have a better understanding, in the right plot of Figure.\ref{fig:legendre2d_gradient}, we also present recovery probability with respect to sparsity $s$ with a fixed number of random samples $N=35$. Both the two plots show that the use of gradient information can indeed improve the recovery rate, and furthermore, the more gradient information is included, the better recovery results obtained.

\begin{figure}[!ht]
\centering
\includegraphics[width=0.48\textwidth,height=0.24\textheight]{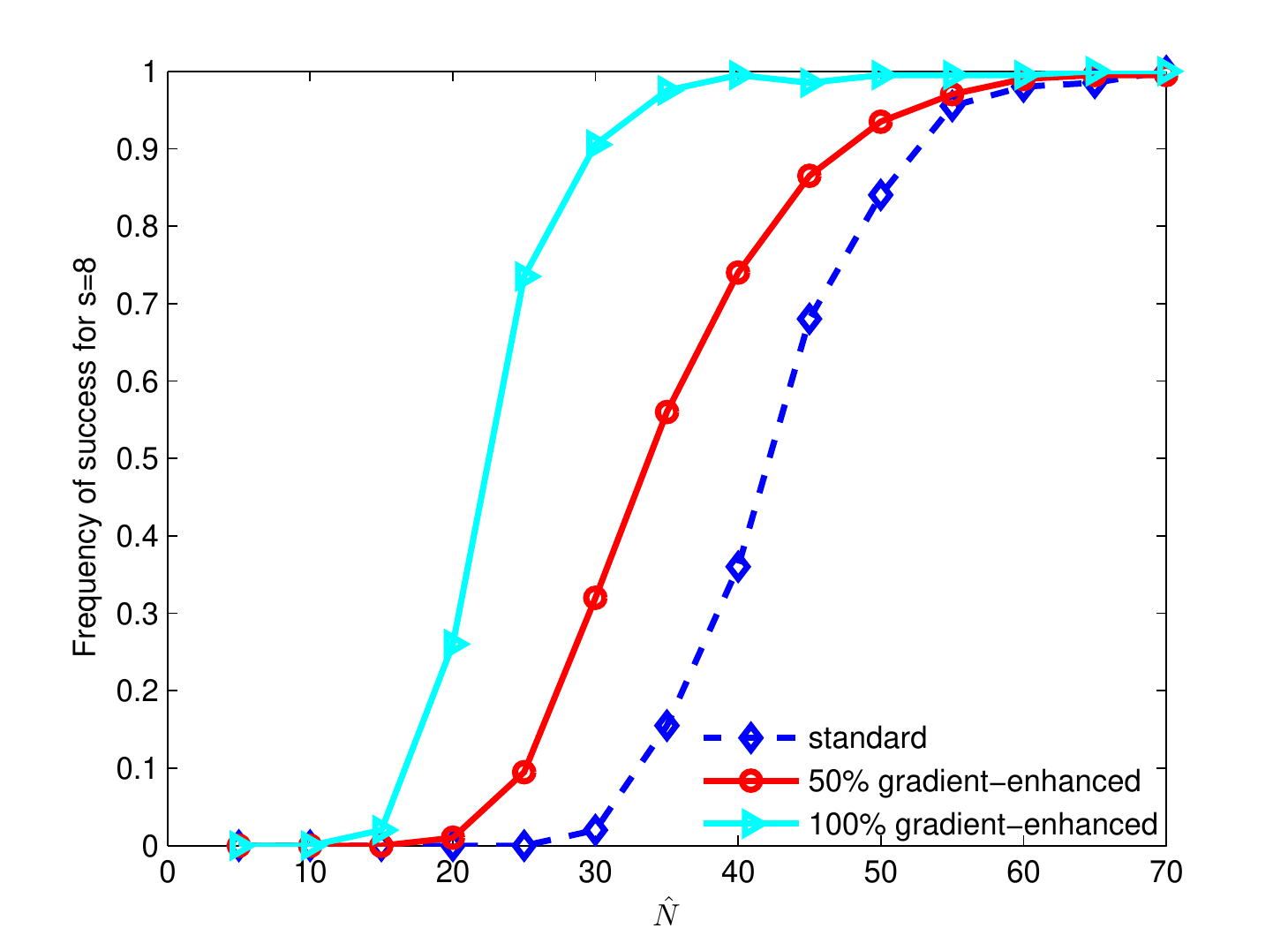}\quad
\includegraphics[width=0.48\textwidth,height=0.24\textheight]{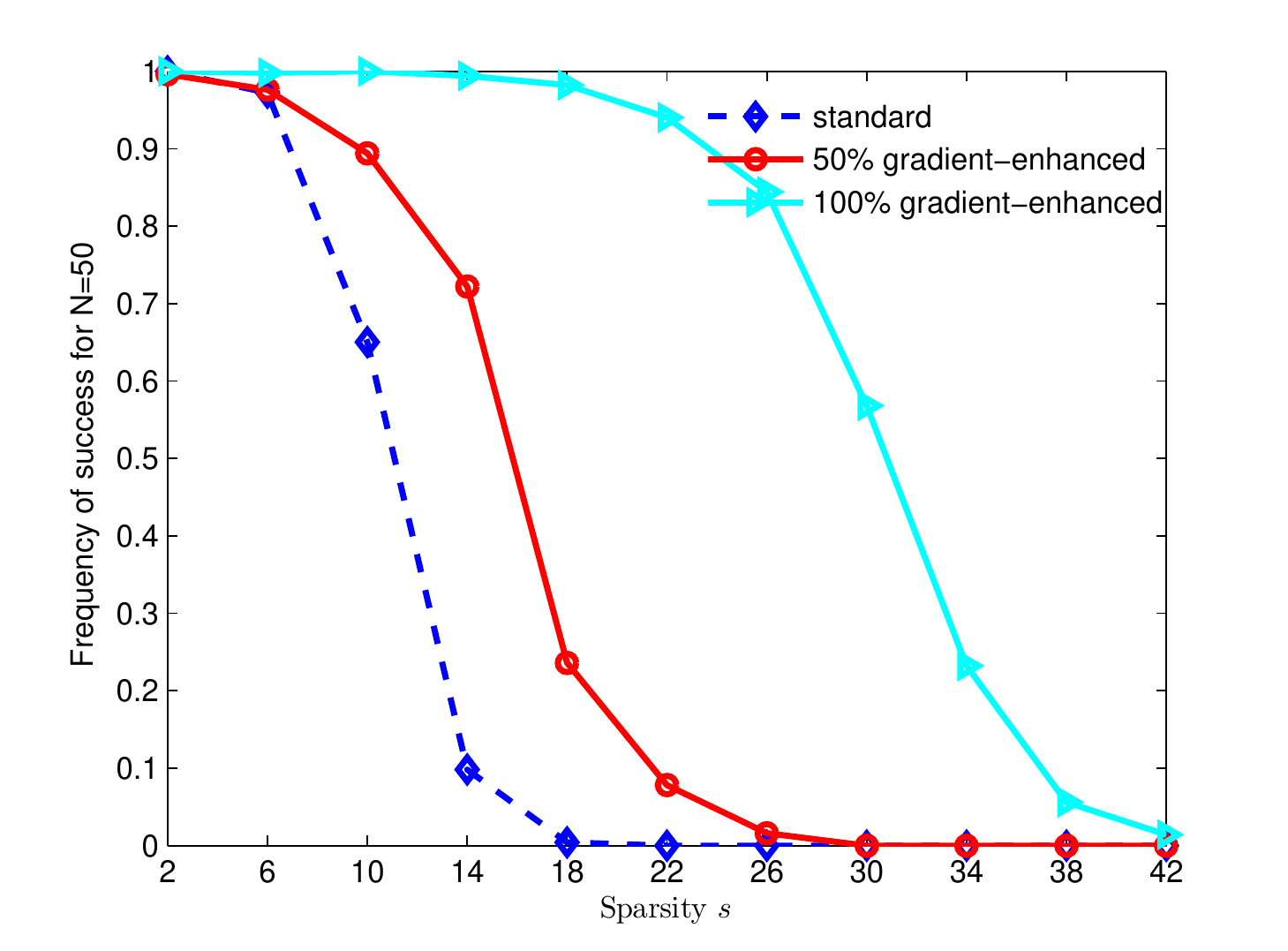}
\caption{Left: Recovery probability against number of samples, $s = 8$; Right: Recovery
probability against sparsity, $N = 50$. Two dimensional tests with $d = 2, n = 20$. Legendre polynomial and Chebyshev samples.}\label{fig:legendre2d_gradient}
\end{figure}

We now consider the 10-dimensional case. In Figure.\ref{fig:legendre10d_gradient} (Left), we show the recovery probability against number of sample points $N$ with a fixed sparsity $s=6$ and the right plot, we present the recovery probability with respect to sparsity with a fixed number of points $N=50$. In this example, we test the $10\%$ and $20\%$ gradient-enhanced approach, meaning that only one or two partial derivatives are involved in the $\ell_1$ minimization. Once again, better performance can be observed when gradient information is included.

\begin{figure}[!ht]
\centering
\includegraphics[width=0.48\textwidth,height=0.24\textheight]{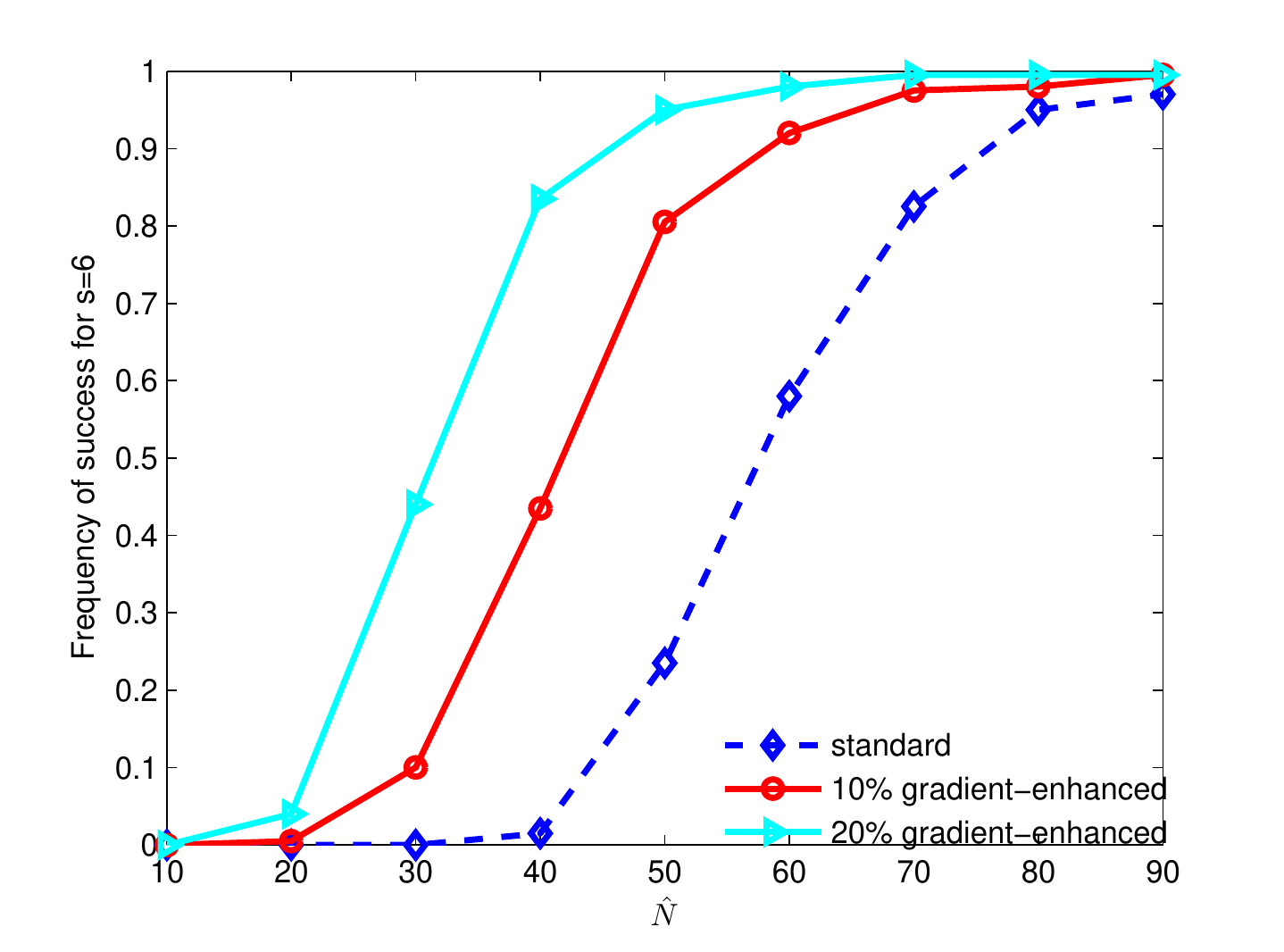}\quad
\includegraphics[width=0.48\textwidth,height=0.24\textheight]{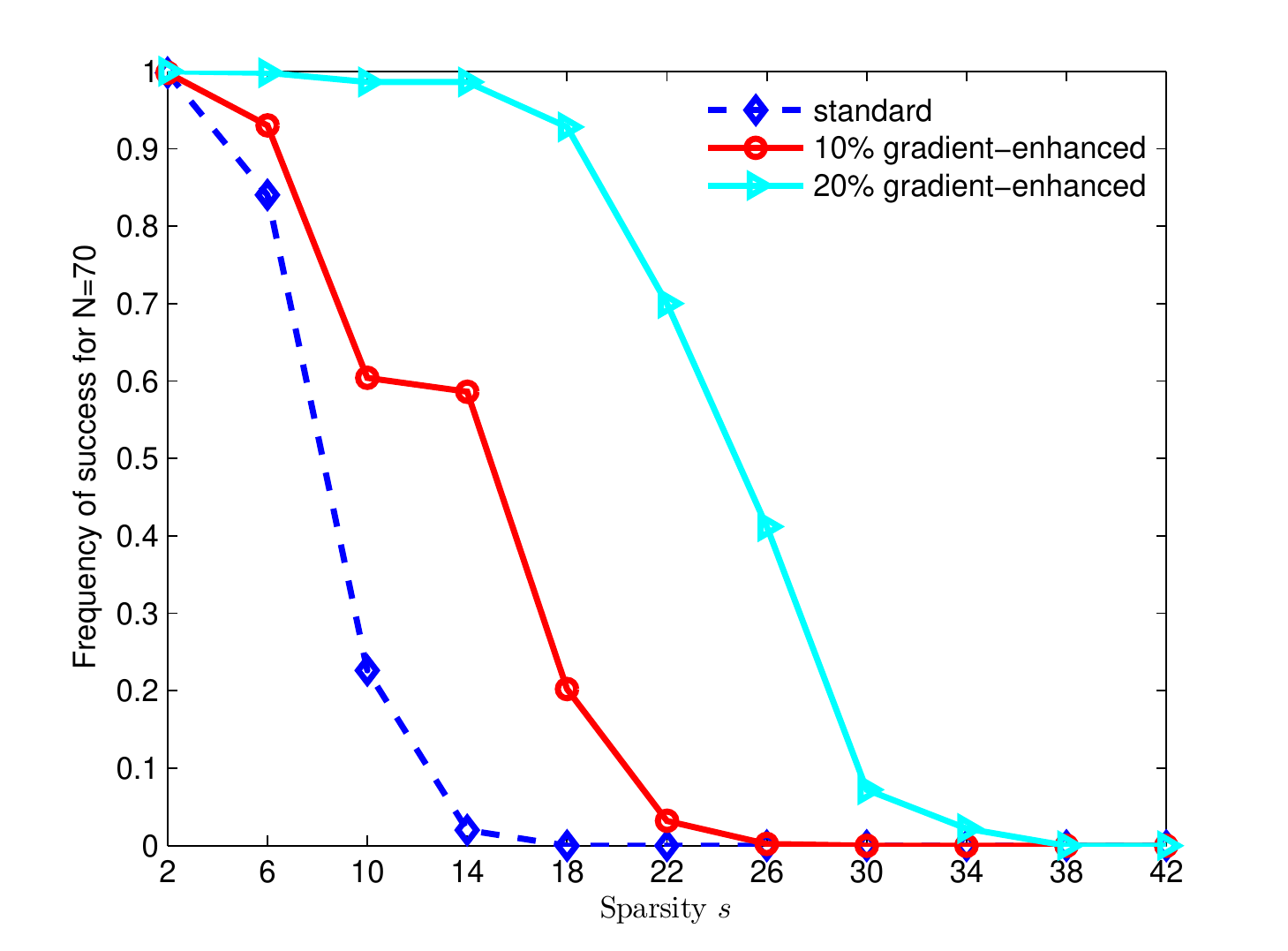}
\caption{Left: Recovery probability against number of samples, $s = 6$; Right: Recovery
probability against sparsity, $N = 70$. Two dimensional tests with $d = 10, n =3$. Legendre polynomial and Chebyshev samples.}\label{fig:legendre10d_gradient}
\end{figure}

\subsection{Applications to function approximations}
In this section, we demonstrate the utility of using gradient data to build PCE approximations for different kind of test functions defined as follows.

Sphere function:
\begin{equation*}
%f_1(x)=\exp\left(\sum_{i=1}^da_i^2(x_i-w_i)^2\right), \quad a_i=\exp\left(-\frac{6i}{d}\right)\ \ w_i=0.5
f_1(x)=\sum_{i=1}^dx_i^2,
\end{equation*}

Gaussian\ function:
\begin{equation*}
f_2(x)=\exp\bigg(-\sum_{i=1}^{d}0.01(1/2(x_{i}+1)-0.375)^2\bigg),
\end{equation*}

Sinusoids\ function:
\begin{equation*}
f_3(x)=\sum_{i=1}^d0.3+\sin(16/15x_i-0.7)+\sin^2(16/15x_i-0.7).
\end{equation*}

In Figure.\ref{fig:legendref1_gradient}, we consider to approximate the sphere function with Legendre polynomial chaos and random evaluations using the $\ell_1$ approach. The left plot shows the root-mean-square-error (RMSE) against the number of sample points $N$ for the two dimensional case (with $n=20$ and $M=231$), while the right plot presents the RMSE against the number of sample points for the 10-dimensional case (with $n=3$ and $M=286$). In both cases, it is clear shown that the use of gradient information can dramatically enhance the approximation accuracy. Similar tests are done for the Gaussian and Sinusoids functions, and the numerical results are presented in Figure. \ref{fig:legendref2_gradient} and Figure. \ref{fig:legendref3_gradient}, respectively.

\begin{figure}[!ht]
\centering
\includegraphics[width=0.48\textwidth,height=0.24\textheight]{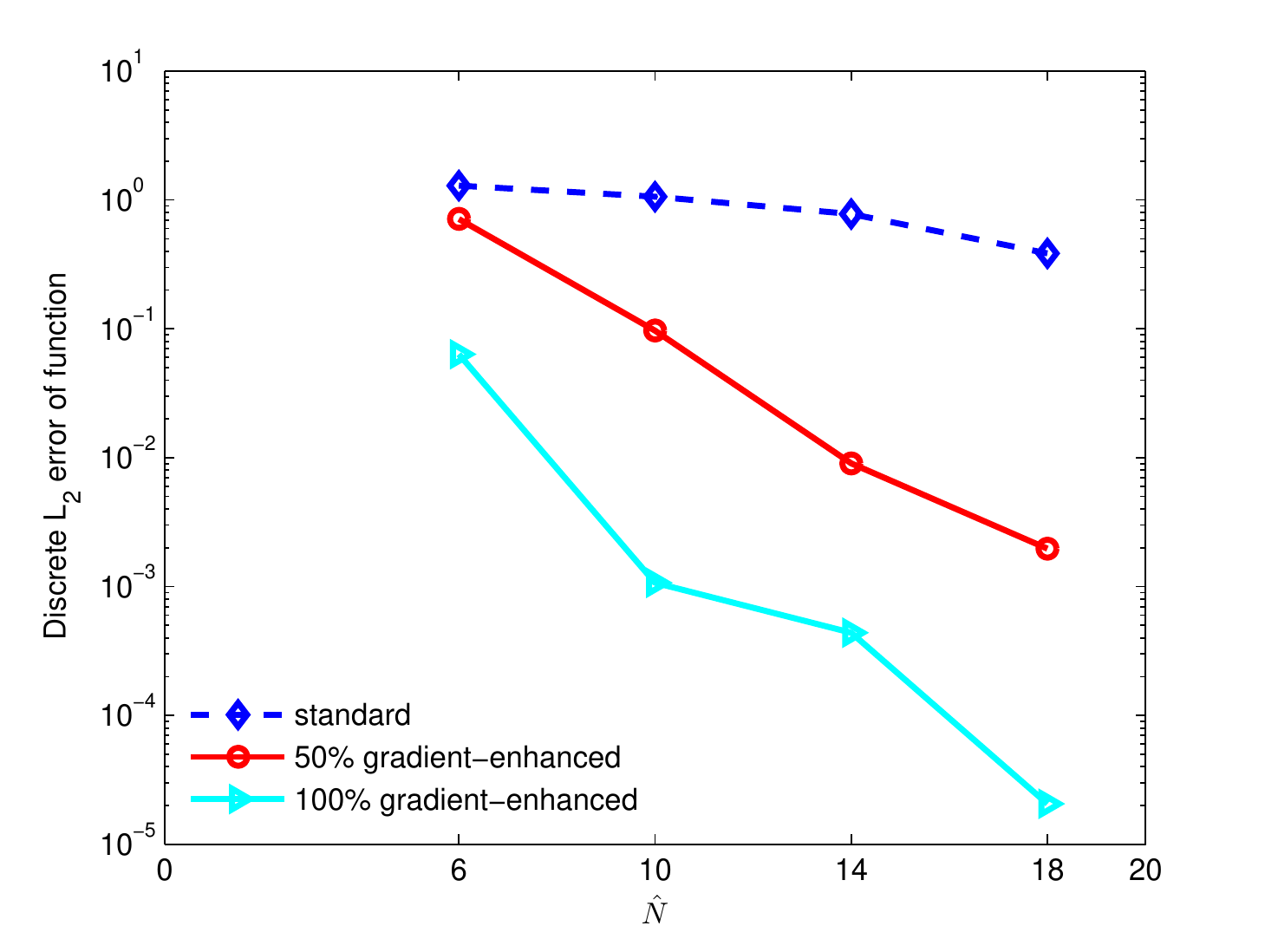}\quad
\includegraphics[width=0.48\textwidth,height=0.24\textheight]{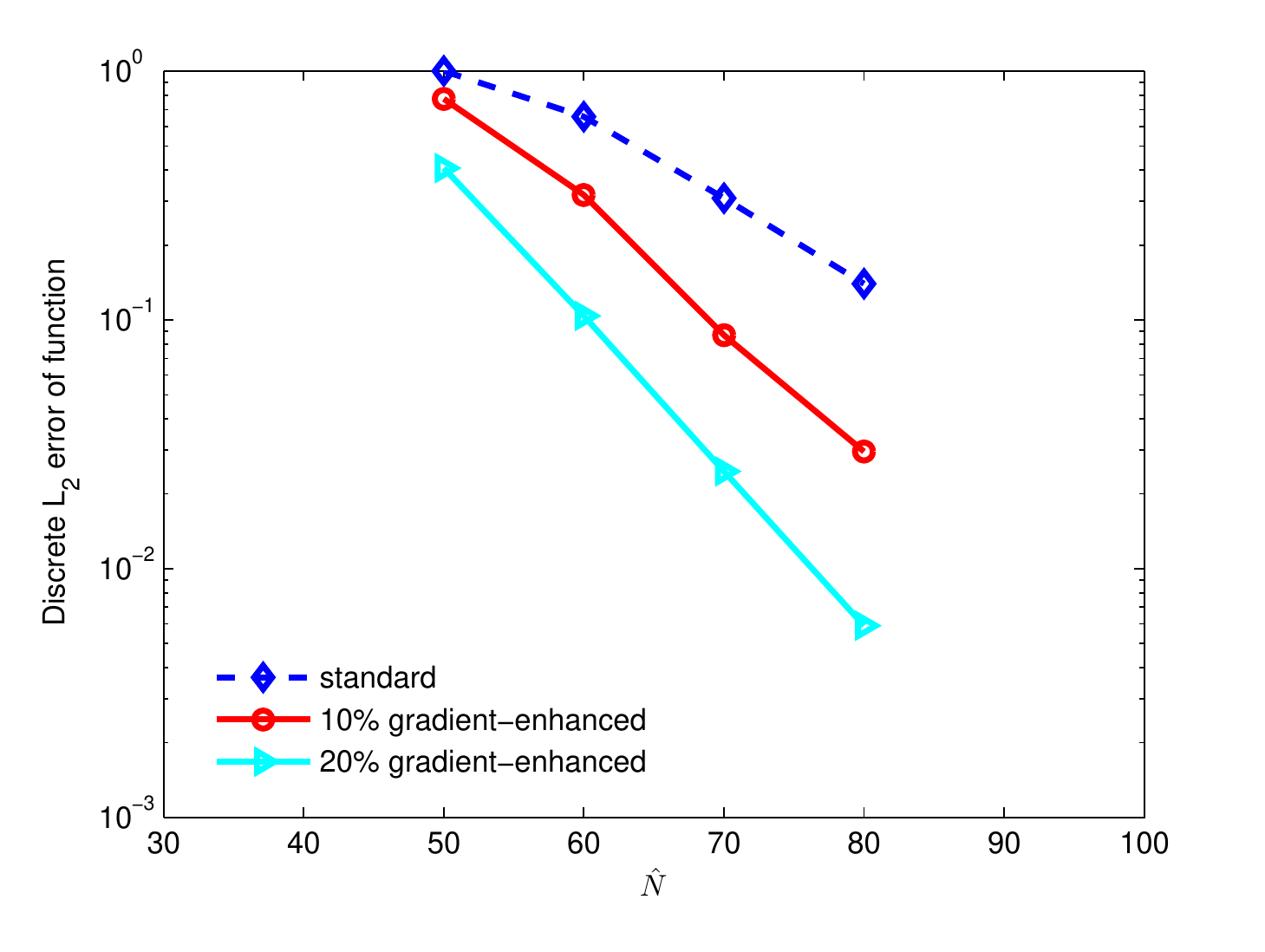}
\caption{Discrete $L_2$ error against number of samples with random points of $f_1(x)$.  Legendre polynomial with Chebyshev sampling.
Left: $d = 2, n = 20$. Right: $d = 10, n = 3$.}\label{fig:legendref1_gradient}
\end{figure}

\begin{figure}[!ht]
\centering
\includegraphics[width=0.48\textwidth,height=0.24\textheight]{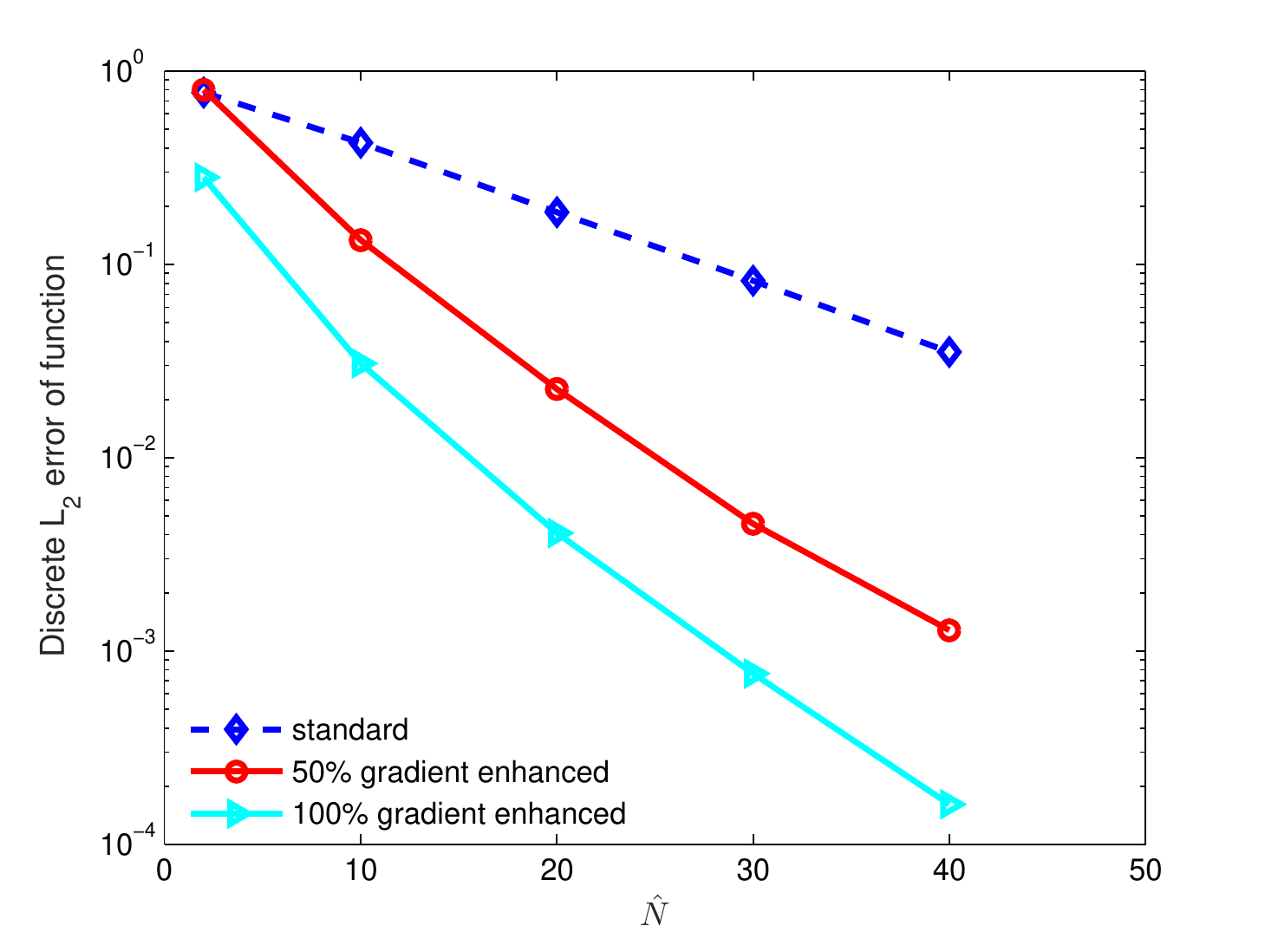}\quad
\includegraphics[width=0.48\textwidth,height=0.24\textheight]{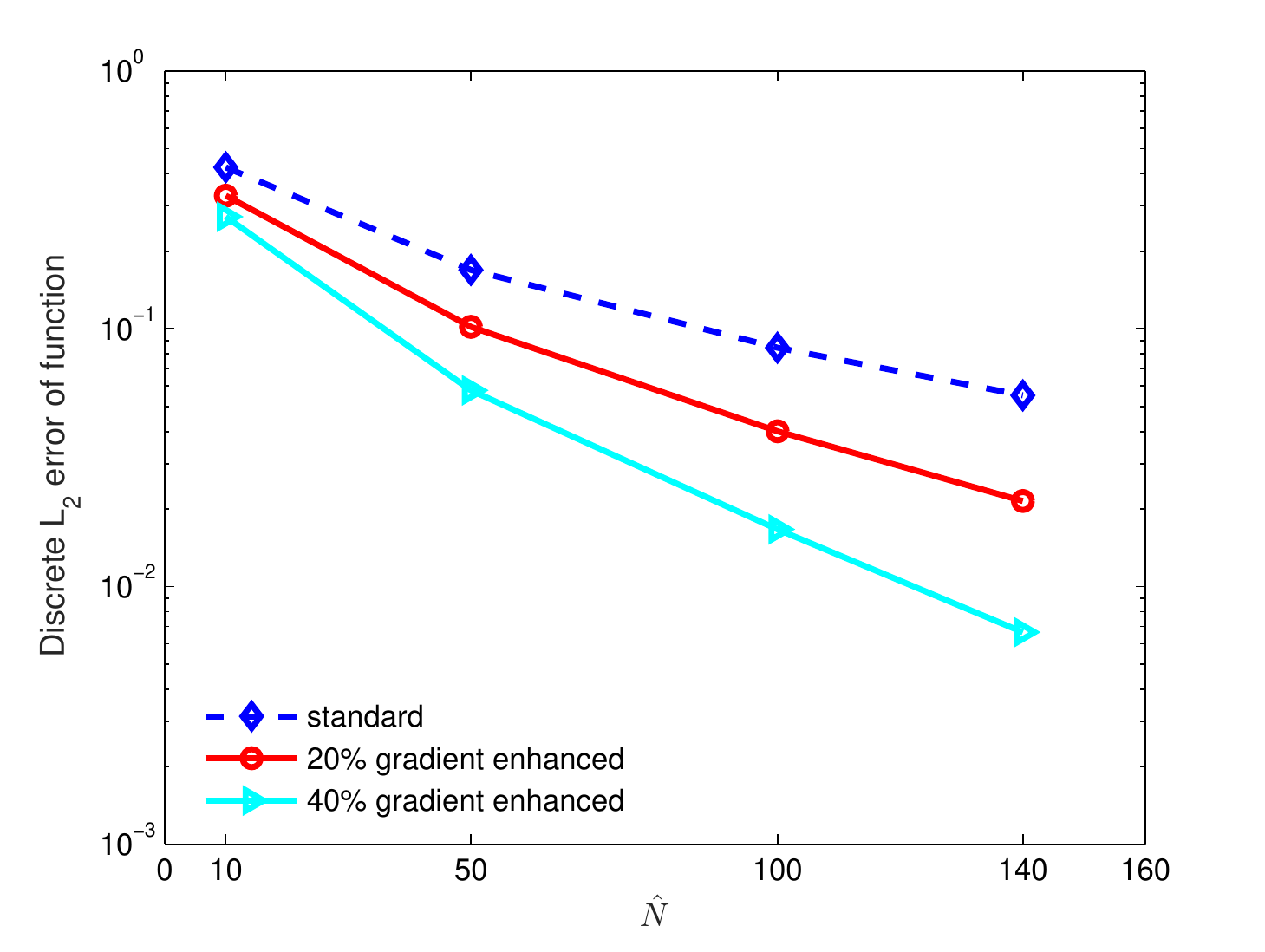}
\caption{ Discrete $L_2$ error against number of samples with random points of $f_2(x)$.  Legendre polynomial and Chebyshev samples.
Left: $d = 2, n = 20$. Right: $d = 6, n = 5$.}\label{fig:legendref2_gradient}
\end{figure}

\begin{figure}[!ht]
\centering
\includegraphics[width=0.48\textwidth,height=0.24\textheight]{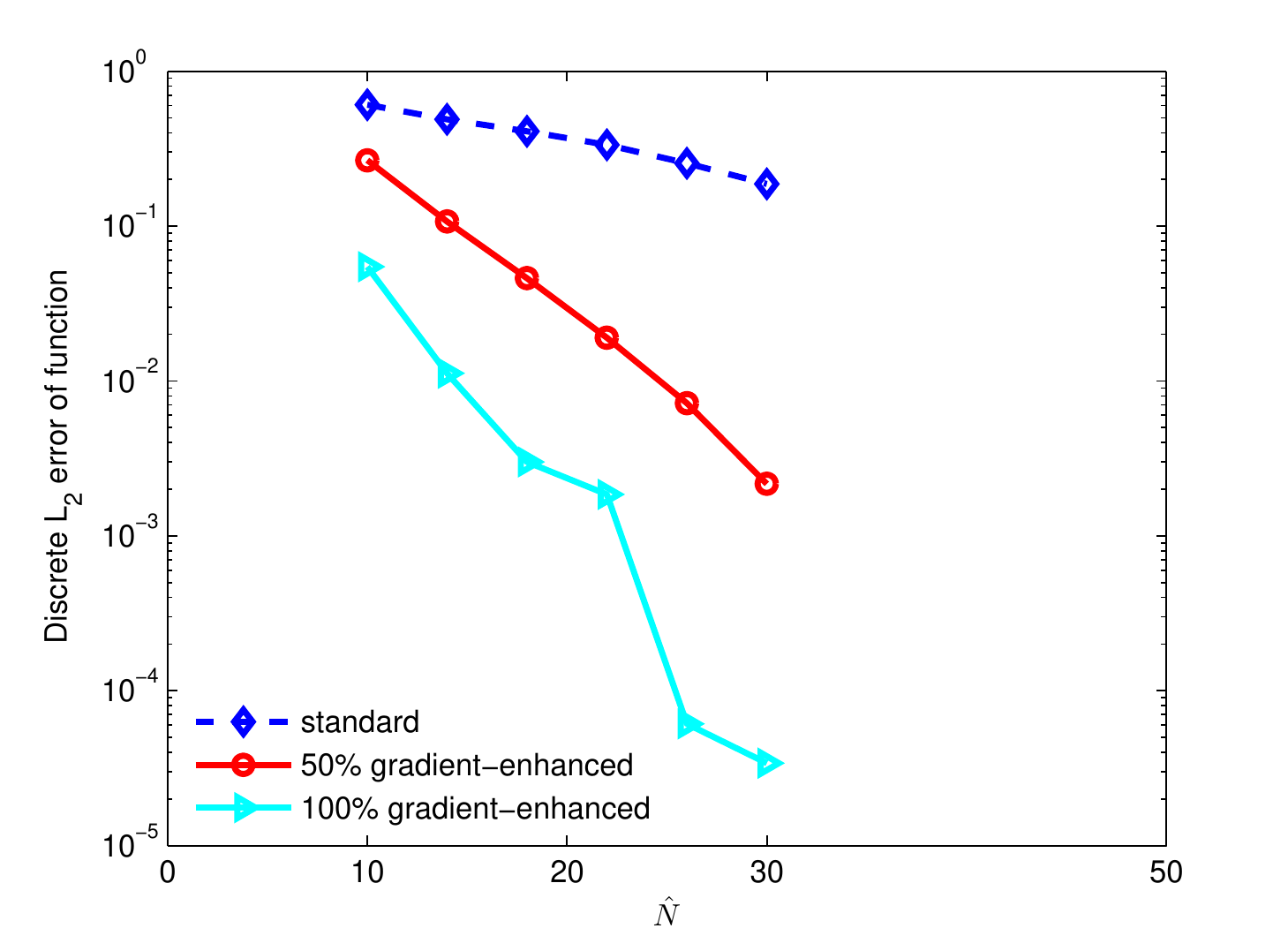}\quad
\includegraphics[width=0.48\textwidth,height=0.24\textheight]{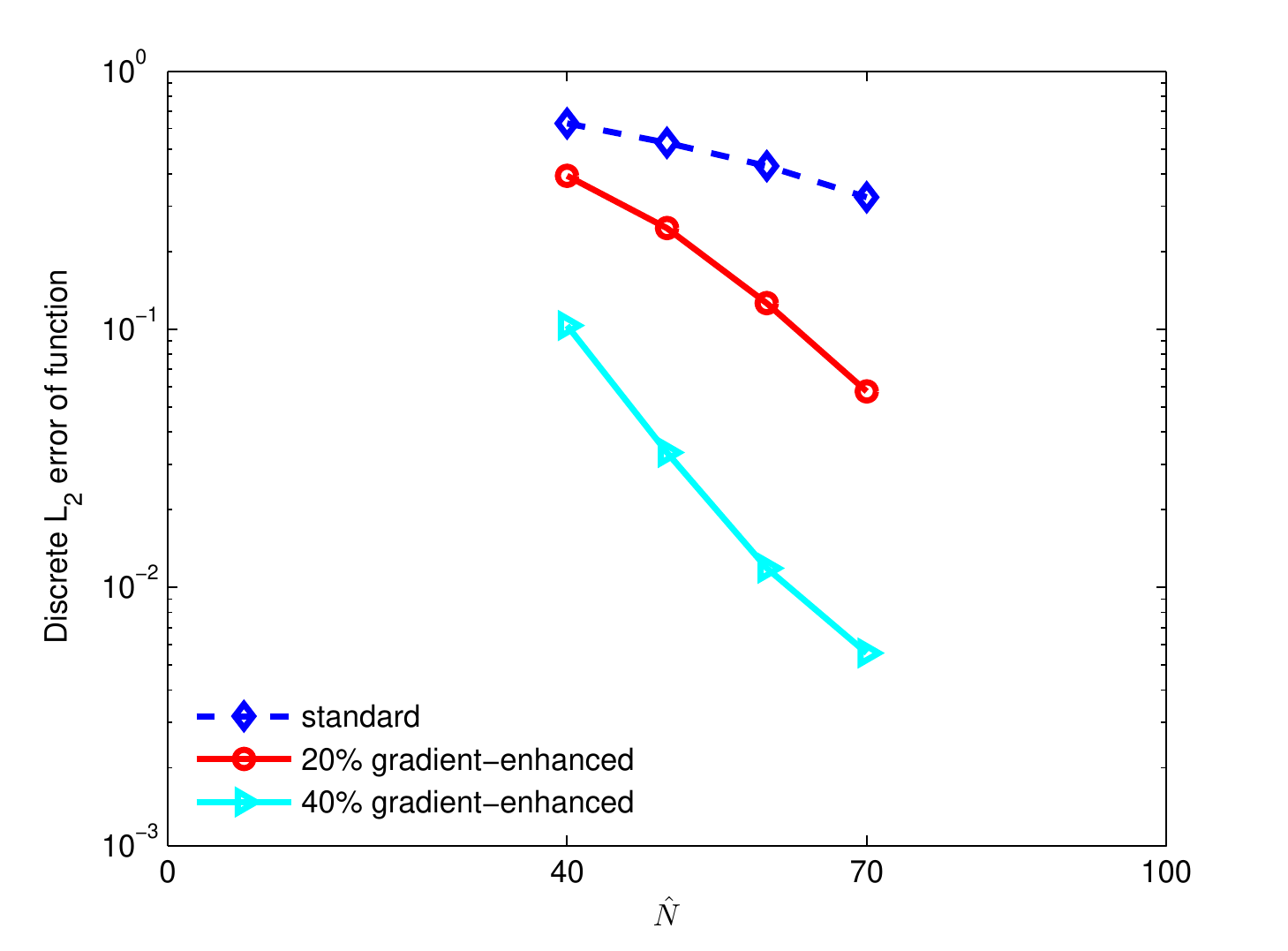}
\caption{Discrete $L_2$ error against number of samples with random points of $f_3(x)$.  Legendre polynomial and Chebyshev samples.
Left: $d = 2, n = 20$. Right: $d = 5, n = 6$.}\label{fig:legendref3_gradient}
\end{figure}

\subsection{Elliptic PDE with Random Inputs}
We next consider the following stochastic linear two-dimensional (in spatial) elliptic PDE problem
\begin{equation}\label{eq:PDEmodel}
\begin{cases}
-\nabla\cdot(a(\mathbf{y},\omega)\nabla u(\mathbf{y},\omega))=f(\mathbf{y},\omega)\quad \textmd{in} \ \mathcal{D}\times\Omega,\\
u(\mathbf{y},\omega)=0 \quad \quad \quad \quad \quad \quad \quad \quad \quad \quad \  \textmd{on} \ \partial\mathcal{D}\times\Omega,
\end{cases}
\end{equation}
with spatial domain $\mathcal{D}=[0,1]^{2}$. We take a deterministic load $f(\mathbf{y},\omega)=\cos( y_{1})\sin( y_{2})$ for these numerical examples. And construct the random diffusion coefficient $a_{N}(\mathbf{y},\omega)$ with one-dimensional spatial
dependence as in \cite{Babuka_2010SCEPDE}:
\begin{equation*}
\log\big(a_{N}(\mathbf{y},\omega)-0.5\big)=1+\xi_{1}(\omega)\Big(\sqrt{\pi}L/2\Big)^{1/2}+\sum_{i=2}^{d} \zeta_{i}g_{i}(\mathbf{y})\xi_{i}(\omega),
\end{equation*}
where
\begin{equation*}
\zeta_{i}:=(\sqrt{\pi}L)^{1/2}\exp\Big(\frac{-(\lfloor\frac{i}{2}\rfloor\pi L)^{2}}{8}\Big), \quad   i>1
\end{equation*}
and

\begin{equation*}
g_{i}(\mathbf{y}):= \begin{cases}
\sin\left(-\lfloor\frac{i}{2}\rfloor\pi  y_{1}\right),  \,\,\,  \textmd{i} \ \textmd{even},\\[12pt]
\cos\left(-\lfloor\frac{i}{2}\rfloor\pi  y_{1}\right),  \,\,\,  \textmd{i} \ \textmd{odd}.
\end{cases}
\end{equation*}

Here $\{\xi_{i}\}^{d}_{i=1}$ are uniformly distributed on the interval $[-1,1]$. We assume that $\xi_i$ are mutually independent from each other. Hence, a family of Legendre polynomials is used to approximate the quantities of interest of $\bs \xi$. Here $\mathbf{y}$ represents the physical domain. The random diffusion coefficient $a_{N}(\mathbf{y},\omega)$ used here only depends on $y_1$.  For $y_{1}\in[0,1]$, let $L=1/12$ be a desired physical correlation length for $a(\mathbf{y},\omega)$. The deterministic elliptic equation are solved by a standard finite element method with a fine mesh.

The convergence rates are shown in Fig. \ref{fig:lcPDE3d_gradient} and Fig. \ref{fig:lcPDE10d_gradient} for a low dimensional case ($d=3,n=10$) and a high dimensional case ($d=10,n=4$), respectively. In the numerical tests, we employ a FEM solver as the deterministic solver and the Monte Carlo method with $6000$ samples are used to get the reference mean and standard deviation of the solution. The gradient information is obtained by solving the adjoint equation as in \cite{jakeman}.  Finally, the numerical error of our approach for the mean and standard deviation are presented. We learn again in the pictures that the gradient-enhanced approach performs much better than the standard $\ell_1$ approach.

\begin{figure}[!ht]
\centering
\includegraphics[width=0.48\textwidth,height=0.24\textheight]{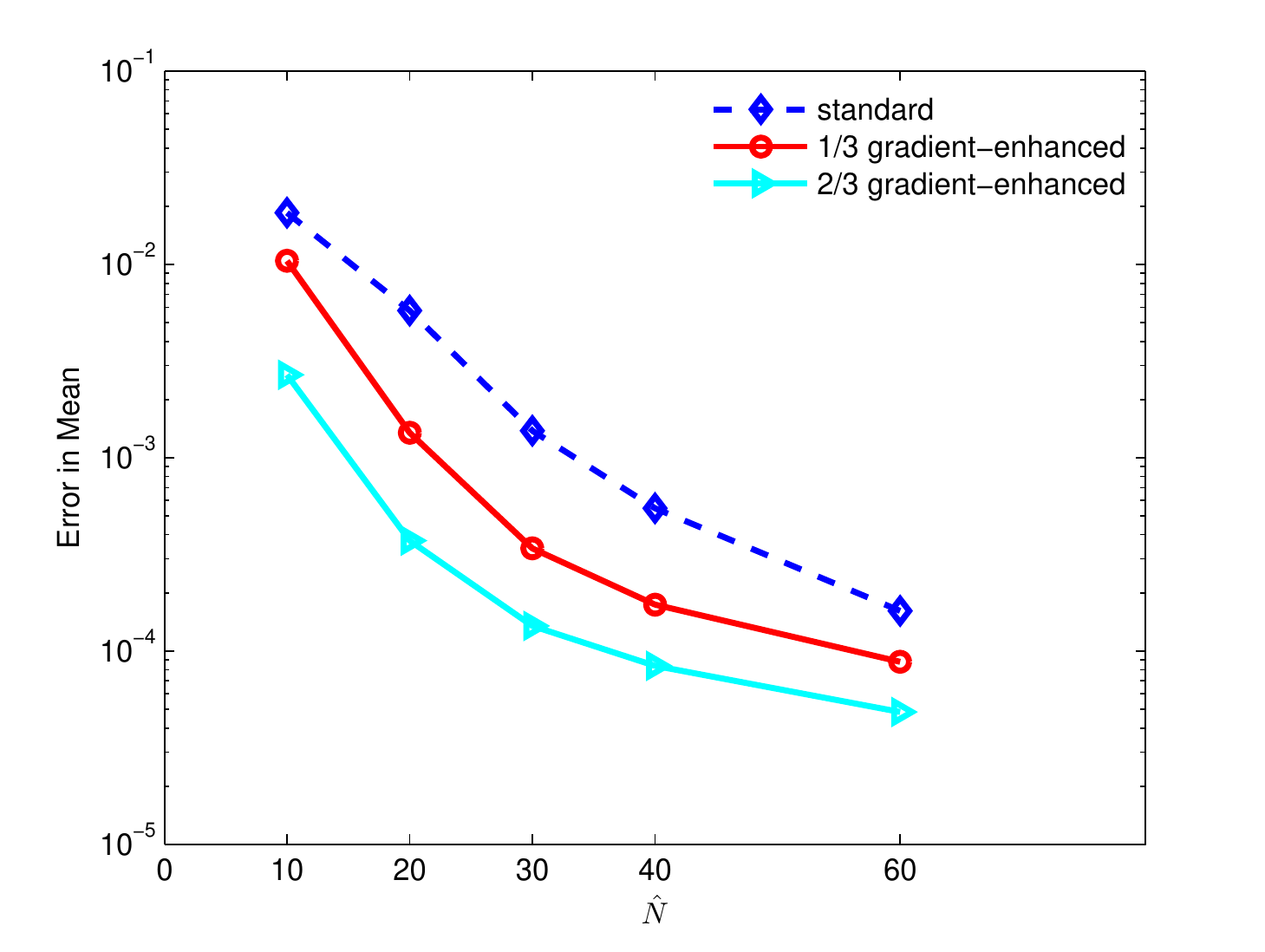}\quad
\includegraphics[width=0.48\textwidth,height=0.24\textheight]{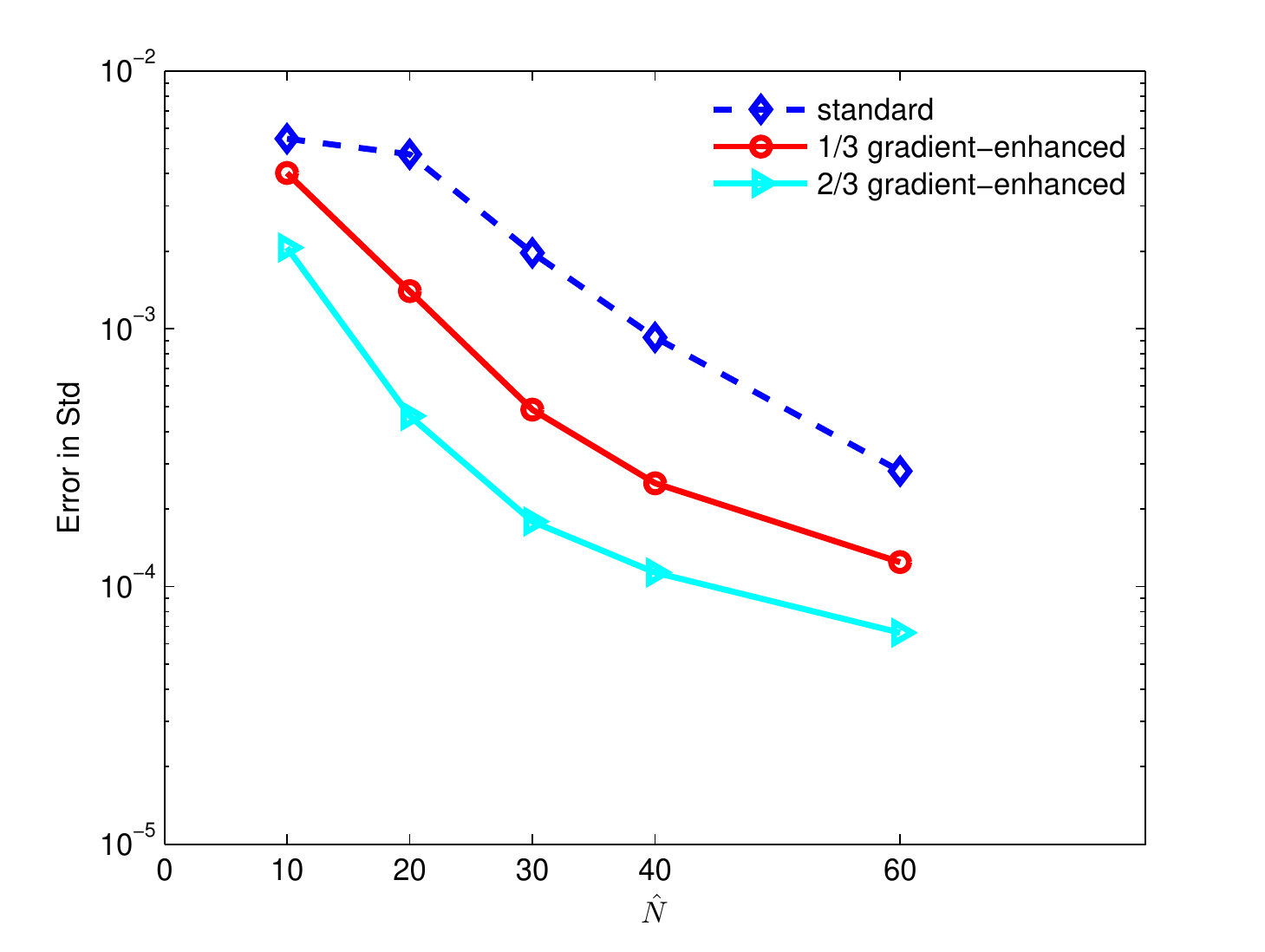}
\caption{Error in $\ell_2$ norm of the mean and variance between the reference and approximation for the various gradient-enhanced method as a function of the number of samples $N$. $d = 3, n = 10.$}\label{fig:lcPDE3d_gradient}
\end{figure}

\begin{figure}[!ht]
\centering
\includegraphics[width=0.48\textwidth,height=0.24\textheight]{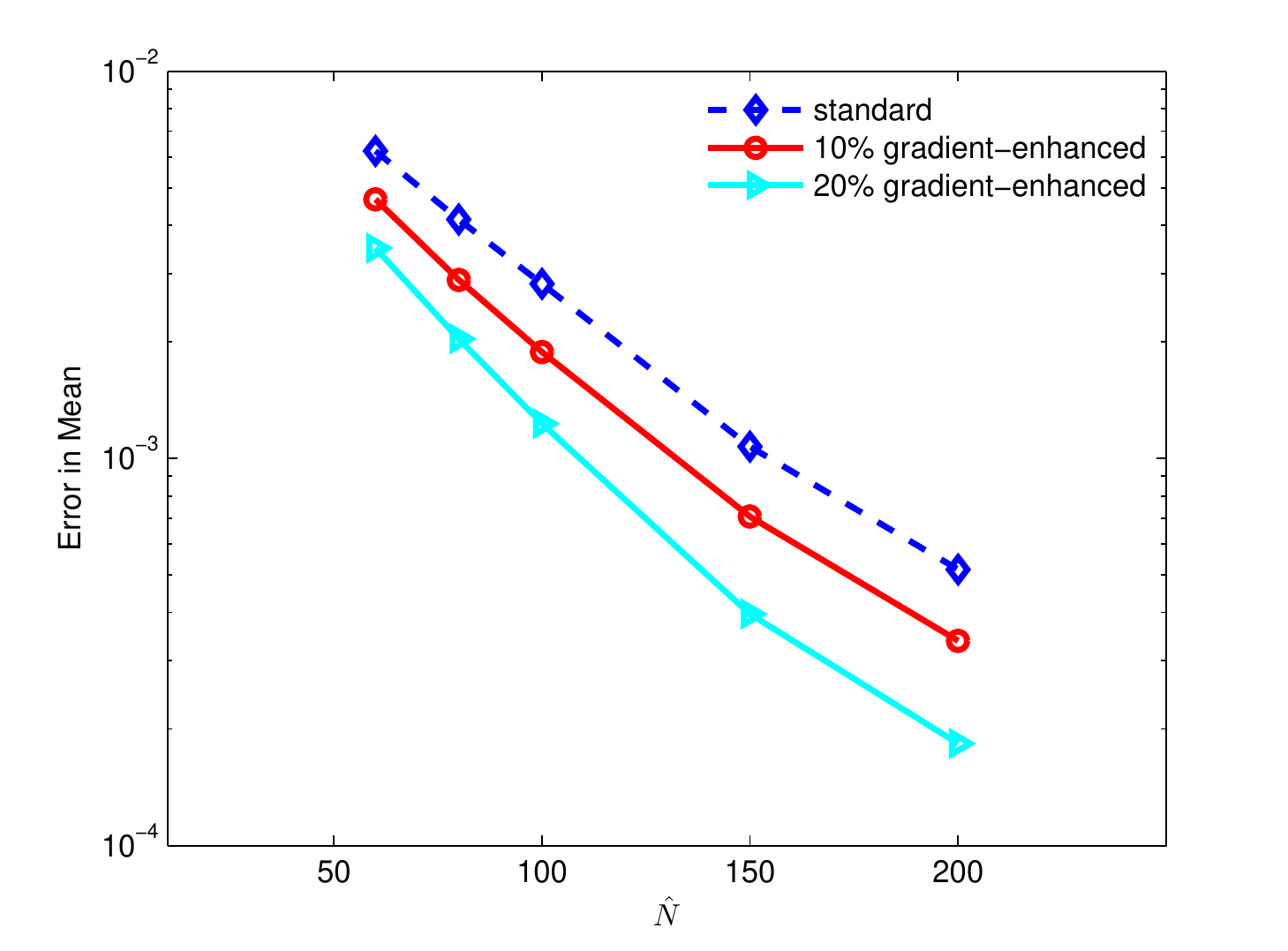}\quad
\includegraphics[width=0.48\textwidth,height=0.24\textheight]{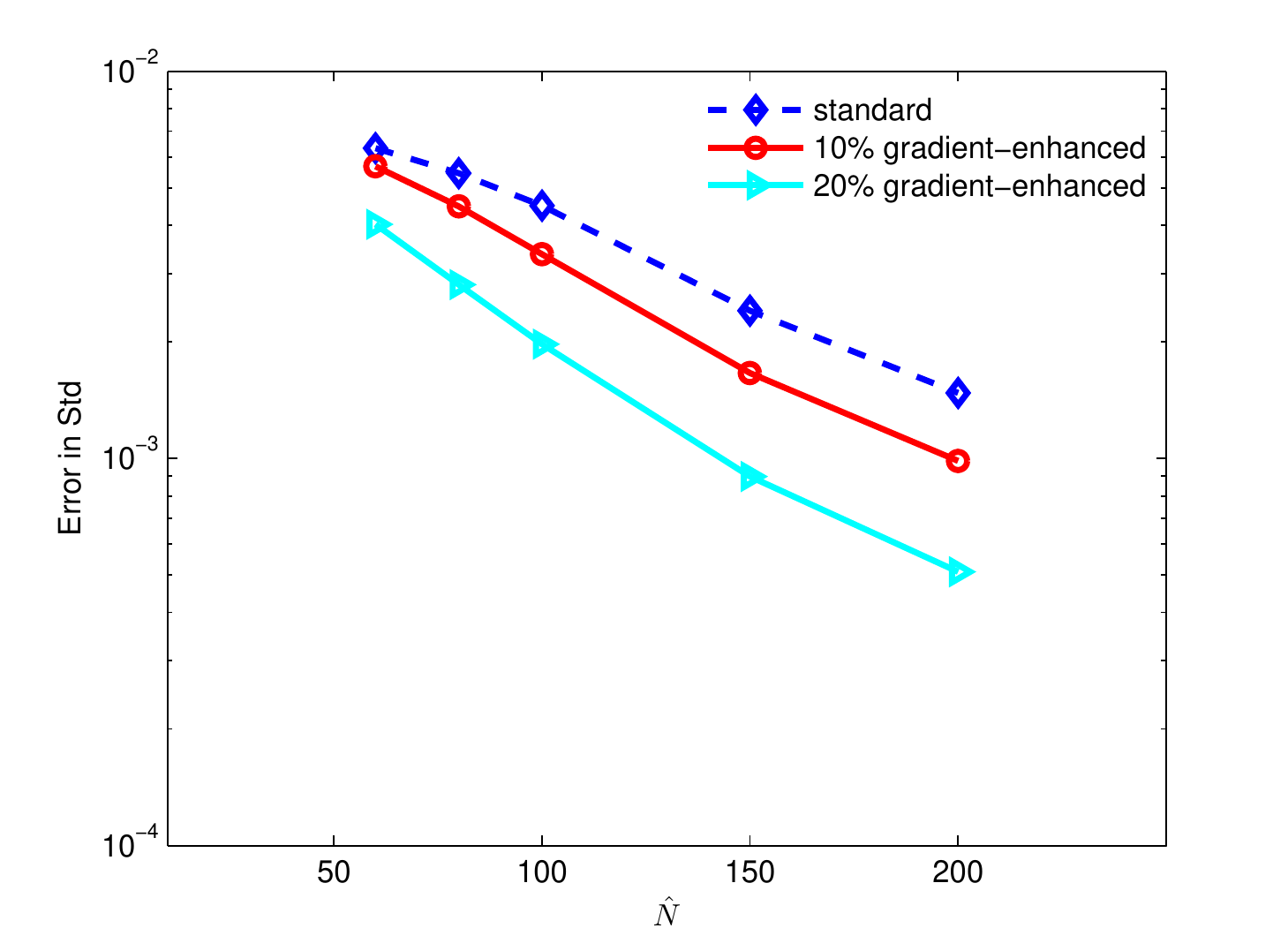}
\caption{Error in $\ell_2$ norm of the mean and variance between the reference and approximation for the various gradient-enhanced method as a function of the number of samples $N$. $d = 10, n =4.$}\label{fig:lcPDE10d_gradient}
\end{figure}

\section{Conclusion}

In this work, we present a general framework for the gradient-enhanced $\ell_1$-minimization for constructing the sparse polynomial chaos expansions. By designing appropriate pre-conditioners to the measure matrix, we show the inclusion of derivative information can indeed improve the recovery property. And the framework is quite general and it applies to both problems with bounded random input and unbounded random input. Several numerical examples are presented to support the theoretical finding.

\section*{Acknowledgments}
We would like to thank Prof. Dongbin Xiu from Ohio State university for introducing this topic to us a few years ago, and also for his very helpful comments.

\appendix

\section{Proofs}

In this appendix we collect results which imply the results given by theorem \ref{th:main}. These are essentially well-known results in the theory of orthogonal polynomials. Our analysis uses these well-known results in fairly straightforward ways.

Let $p_n^{(\alpha, \beta)}(x)$ be degree-$n$ Jacobi polynomial orthonormal under the Jacobi probability weight $\rho^{(\alpha, \beta)}(x)$ defined in (\ref{eq:jacobi-density}), i.e.,
\begin{align*}
\int_{-1}^1 p_n^{(\alpha,\beta)}(x) p_m^{(\alpha,\beta)}(x) \rho^{(\alpha, \beta)}(x) \dx = \delta_{n,m},
\end{align*}
The following identity holds:
\begin{align}\label{eq:jacobi-derivative}
 \frac{\mathd }{\mathd x} p_n^{(\alpha,\beta)}(x) &= c(n,\alpha, \beta) p_{n-1}^{(\alpha+1,\beta+1)}(x) , \\
 c^2(n,\alpha, \beta) &= n (n+\alpha + \beta + 1) \frac{(\alpha + \beta+1)(\alpha + \beta + 2)}{4 (\alpha+1)(\beta+1)}
\end{align}
The following result is also critical for us:
\begin{lemma}[\cite{Jacobi}]\label{eq:jacobi-bound-le}For all Jacobi weight functions $\rho^{(\alpha, \beta)}(x)$ with $\alpha\geq -\frac{1}{2}$ and $\beta\geq -\frac{1}{2}$, the following inequalities hold,
\begin{align}\label{eq:jacobi-bound}
  \sup_{x \in [-1,1]} \frac{\rho^{(\alpha, \beta)}(x)}{\rho_c(x)} \left[ p_n^{(\alpha, \beta)}(x) \right]^2 \leq 2 e \left( 2 + \sqrt{\alpha^2 + \beta^2}\right),
\end{align}
uniformly in $n, \alpha, \beta$.
\end{lemma}

Therefore, consider performing $\ell_1$ optimization with derivative evaluations using Jacobi polynomials and Chebyshev sampling, as in Section \ref{sec:jacobi}. Each univariate sample $z$ yields two rows of the design matrix, whose entries are
\begin{align*}
 \left( \begin{array}{c} p_n^{(\alpha,\beta)}(z) \\ \ddx{x} p_n^{(\alpha,\beta)}(z) \end{array} \right) = \left( \begin{array}{c} p_n^{(\alpha,\beta)}(z) \\ c(n,\alpha,\beta) p_{n-1}^{(\alpha+1,\beta+1)}(z) \end{array} \right)
\end{align*}
Since $z$ is distributed according to the Chebyshev measure, we need to precondition these rows in order to keep mean isotropy:
\begin{align*}
 \left( \begin{array}{c} \sqrt{\frac{\rho^{(\alpha,\beta)}(z)}{\rho_c(z)}} p_n^{(\alpha,\beta)}(z) \\ \frac{1}{c(n,\alpha,\beta)} \sqrt{\frac{\rho^{(\alpha+1,\beta+1)}(z)}{\rho_c(z)}} \ddx{x} p_n^{(\alpha,\beta)}(z) \end{array} \right) = \left( \begin{array}{c} \sqrt{\frac{\rho^{(\alpha,\beta)}(z)}{\rho_c(z)}} p_n^{(\alpha,\beta)}(z) \\ \sqrt{\frac{\rho^{(\alpha+1,\beta+1)}(z)}{\rho_c(z)}} p_{n-1}^{(\alpha+1,\beta+1)}(z) \end{array} \right)
\end{align*}
Extending this result to the multivariate (tensor-product) case, we can show that the identity \eqref{eq:jacobi_chebyshev} holds for a multi-index $i$. Now we turn to the proof of our main result, Theorem \ref{th:main}.

\subsection{Proof of Theorem \ref{th:main}} For positive $a_0, \ldots, a_{d}$ and $b_0, \ldots, b_d$, the inequality
\begin{align*}
  \frac{ \sum_{j=0}^d a_j}{\sum_{j=1}^d b_j} \leq \max_{j=0, \ldots, d} \frac{a_j}{b_j}
\end{align*}
holds. Define
\begin{align*}
 & \sqrt{a_0(\bs{i}, \bs{x})} \coloneqq \sqrt{\frac{\rho^{(\bs{\alpha},\bs{\beta})}(\bs{x})}{\rho_c(\bs{x})}} \Phi_{\bs{i}}(\bs{x}),&  \sqrt{a_j(\bs{i}, \bs{x})} \coloneqq \sqrt{\frac{\rho^{(\bs{\alpha}+\bs{e}_j,\bs{\beta}+\bs{e}_j)}(\bs{x})}{\rho_c(\bs{x})}} \ppx{x_j} \Phi_{\bs{i}}(\bs{x}),  \\
 & \sqrt{b_0(\bs{i}, \bs{x})} \coloneqq 1, &
  \sqrt{b_j(\bs{i}, \bs{x})} \coloneqq c\left( i_j, \alpha_j, \beta_j \right), \quad j = 1, \ldots, d
\end{align*}
Note that
\begin{align*}
  \gamma_j &\coloneqq \frac{b_j(\bs{i}, \bs{x})}{\sum_{k=0}^d b_k(\bs{i}, \bs{x})} \in [0, 1], \quad \sum_{j=0}^d \gamma_j = 1,
\end{align*}
so that the $\gamma_j$ are convex weights. Then we can rewrite
\begin{align*}
\mu\left( \bs{\Phi}\right) &= \sup_{\bs{i}, \bs{z}} \frac{a_0(\bs{i}, \bs{z})}{b_0(\bs{i}, \bs{z})} \\
\mu\left( \bs{\widehat{\Phi}}\right) &= \sup_{\bs{i}, \bs{z}} \frac{ \sum_{j=0}^d a_j(\bs{i}, \bs{z})}{\sum_{j=0}^d b_j(\bs{i}, \bs{z})} = \sup_{\bs{i},\bs{z}} \sum_{j=0}^d \frac{a_j(\bs{i},\bs{z})}{\sum_{k=0}^d b_j(\bs{i},\bs{z})}
\end{align*}
Note that
\begin{align*}
  \frac{a_0}{b_0} &= \frac{\rho^{(\bs{\alpha},\bs{\beta})}(\bs{x})}{\rho_c(\bs{x})} \Phi^2_{\bs{i}}(\bs{x}),\\
  \frac{a_0}{\sum_{k=0}^d b_k} &= \frac{b_0}{\sum_{k=0}^d b_k} \frac{\rho^{(\bs{\alpha},\bs{\beta})}(\bs{x})}{\rho_c(\bs{x})} \Phi^2_{\bs{i}}(\bs{x}) = \gamma_0\frac{\rho^{(\bs{\alpha},\bs{\beta})}(\bs{x})}{\rho_c(\bs{x})} \Phi^2_{\bs{i}}(\bs{x})  \\
  \frac{a_j}{\sum_{k=0}^d b_k} &\stackrel{\eqref{eq:jacobi-derivative}}{=} \frac{b_j}{\sum_{k=0}^d b_k} \frac{\rho^{(\bs{\alpha}+\bs{e}_j,\bs{\beta}+\bs{e}_j)}(\bs{x})}{\rho_c(\bs{x})} \Phi^2_{\bs{i}-\bs{e}_j}(\bs{x}) = \gamma_j \frac{\rho^{(\bs{\alpha}+\bs{e}_j,\bs{\beta}+\bs{e}_j)}(\bs{x})}{\rho_c(\bs{x})} \Phi^2_{\bs{i}-\bs{e}_j}(\bs{x}) \hskip 10pt (1 \leq j \leq d)
\end{align*}
Therefore,
\begin{align*}
\mu\left( \bs{\widehat{\Phi}}\right)  &= \sup_{\bs{i},\bs{z}} \sum_{j=0}^d \gamma_j \frac{\rho^{(\bs{\alpha}+\bs{e}_j,\bs{\beta}+\bs{e}_j)}(\bs{z})}{\rho_c(\bs{z})} \Phi^2_{\bs{i}-\bs{e}_j}(\bs{z}) \\
&\stackrel{\eqref{eq:jacobi-bound}}{\leq} \sum_{j=0}^d \gamma_j \prod_{k=1}^d 2 e \left( 2 + \sqrt{(\alpha_j+\delta_{k,j})^2 + (\beta_j+\delta_{k,j})^2} \right)\\
&= \left[ \prod_{k=1}^d 2 e \left( 2 + \sqrt{(\alpha_j)^2 + (\beta_j)^2} \right)\right] \sum_{j=0}^d \gamma_j \frac{2 + \sqrt{(\alpha_j+1)^2 + (\beta_j+1)^2} }{2 + \sqrt{\alpha_j^2 + \beta_j^2}} \\
&\leq \left[ \prod_{k=1}^d 2 e \left( 2 + \sqrt{(\alpha_j)^2 + (\beta_j)^2} \right)\right] \max_{j=1, \ldots, d}\frac{2 + \sqrt{(\alpha_j+1)^2 + (\beta_j+1)^2}}{2 + \sqrt{\alpha_j^2 + \beta_j^2}}
\end{align*}
Therefore, \eqref{eq:d-coherence-bound} holds with
\begin{align*}
C =\max_{j=1, \ldots, d} \frac{2 + \sqrt{(\alpha_j+1)^2 + (\beta_j+1)^2}}{2 + \sqrt{\alpha_j^2 + \beta_j^2}}
\end{align*}

The proof for the second statement is similar as in \cite{Peng_2016gradient}: we notice that up to an invertible post-multiplication, $\mathbf{\widehat{\Phi}}$ is a sub-matrix of $\mathbf{\Phi}$ and thus $\mathcal{N}\big(\mathbf{\widehat{\Phi}}\big) \subset \mathcal{N}\big(\mathbf{\Phi}\big).$

\bibliographystyle{plain}
\bibliography{derivative_recovery_new}
%\bibliography{gradient_L1_new}

\end{document}